\documentclass[leqno,12pt]{article}

\usepackage[utf8]{inputenc}
\usepackage{lmodern}

\usepackage[usenames, dvipsnames]{xcolor}
\usepackage{tikz}

\usepackage{amssymb, amsfonts, amsthm, amsmath}
\usepackage[english]{babel}
\usepackage{hyperref}
\hypersetup{
    linktoc=page,
   linkcolor=red,          % color of internal links
  citecolor=blue,        % color of links to bibliography
    filecolor=blue,      % color of file links
   urlcolor=cyan,
    colorlinks=true           % color of external links
}

\usepackage{euscript}
\usepackage{flafter}
\usepackage{pstricks}
\usepackage{pstricks-add}
\usepackage{dsfont}
\usepackage{bm}
\usepackage{variations}

\usepackage[refpage]{nomencl}

\usepackage{makeidx}
\makeindex

\usepackage{mathrsfs}

\evensidemargin\oddsidemargin
\newcommand{\eqnsection}{
\renewcommand{\theequation}{\thesection.\arabic{equation}}
   \makeatletter
   \csname  @addtoreset\endcsname{equation}{section}
   \makeatother}
\eqnsection
\def\r{{\mathbb R}}
\def\P{{\bf P}}
\def\E{{\bf E}}

\def\N{{\mathbb N}}
\def\T{{\mathbb T}}
\def\1{{\mathds{1}}}
\def\mS{{\underline{S}}}

\def\ns{{\mathbf S}}

\newtheorem{theorem}{Theorem}[section]

\newtheorem{fact}[theorem]{Fact}
\newtheorem{lemma}[theorem]{Lemma}
\newtheorem{proposition}[theorem]{Proposition}
\newtheorem{remark}[theorem]{Remark}

\newtheorem{corollary}[theorem]{Corollary}

\newcommand{\floor}[1]{{\left\lfloor #1 \right\rfloor}}

\newcommand{\R}{\mathbb{R}}
\renewcommand{\T}{\mathbb{T}}
\newcommand{\calL}{\mathcal{L}}
\newcommand{\calF}{\mathcal{F}}
\newcommand{\calC}{\mathcal{C}}
\newcommand{\calD}{\mathcal{D}}
\newcommand{\calM}{\mathcal{M}}
\newcommand{\ind}[1]{{\mathbf{1}_{\left\{ #1 \right\}}}}
\newcommand{\indset}[1]{{\mathbf{1}_{ #1 }}}
\newcommand{\norme}[1]{{\left\Vert #1 \right\Vert}}
\renewcommand{\hat}[1]{\widehat{#1}}
\renewcommand{\tilde}[1]{\widetilde{#1}}
\newcommand{\egaldistr}{{\overset{(d)}{=}}}
\newcommand{\wconv}{{\underset{n \to \infty}{\Longrightarrow}}}
\renewcommand{\bar}[1]{\overline{#1}}

\title{On the trajectory of an individual chosen according to supercritical Gibbs measure in the branching random walk}

\author{Xinxin Chen\thanks{ICJ, Université Lyon 1}, Thomas Madaule\thanks{IMT, Université Paul Sabatier}, Bastien Mallein\thanks{LAGA, Université Paris 13}}

\date{\today}

\begin{document}

\maketitle

\begin{abstract}
Consider a branching random walk on the real line. Madaule~\cite{Mad12} showed the renormalized trajectory of an individual selected according to the critical Gibbs measure converges in law to a Brownian meander. Besides, Chen~\cite{Che14+} proved that the renormalized trajectory leading to the leftmost individual at time $n$ converges in law to a standard Brownian excursion. In this article, we prove that the renormalized trajectory of an individual selected according to a supercritical Gibbs measure also converges in law toward the Brownian excursion. Moreover, refinements of this results enables to express the probability for the trajectories of two individuals selected according to the Gibbs measure to have split before time $t$, partially answering a question of~\cite{DSpo88}.
\end{abstract}

\section{Introduction}
\label{sec:introduction}

A \emph{branching random walk} on the real line is a particle system on $\R$ defined as follow : It starts with an unique individual sitting at the origin, forming the $0^\mathrm{th}$ generation of the process. At time~$1$, this individual dies and gives birth to children, which are positioned on $\R$ according to a point process of law $\mathcal{L}$. These children form the $1^\mathrm{st}$ generation. Similarly, at each time $n \in \N$, every individual $z$ of the $(n-1)^\mathrm{st}$ generation dies, giving birth to children positioned according to an independent copy of $\calL$ translated from the position of $z$.

We denote by $\T$ the genealogical tree of the process. Obviously $\T$ is a Galton-Watson tree. For any individual $z \in \T$, we write $|z|$ for the generation to which $z$ belongs and $V(z) \in \R$ for the position of $z$. With these notations, $(V(z),|z|=1)$ has law $\calL$. If $z \in \T$, for all $k \leq |z|$, we denote by $z_k$ the ancestor of $z$ alive at generation $k$. The collection of positions $(V(z),z\in \T)$, together with the genealogical informations, defines the branching random walk.

Throughout this paper, we suppose that the point process law $\calL$ satisfies some integrability conditions. We assume that the Galton-Watson tree is supercritical, or in other words, that
\begin{equation}
  \label{eqn:supercritical}
  \E\left[ \sum_{|z|=1} 1 \right] > 1.
\end{equation}
Note that we do not assume $\P\left(\sum_{|z|=1} 1 = \infty\right) = 0$, hence a given individual may have infinitely many children in this branching random walk. Under assumption \eqref{eqn:supercritical}, the survival set
\begin{equation}
  \label{eqn:survival}
  S := \left\{ \forall n \in \N, \exists z \in \T, |z| \geq n\right\}
\end{equation}
is of positive probability.

Assume also that we are in the so-called ``boundary case" defined in~\cite{BKy05}, i.e.
\begin{equation}
  \label{eqn:boundary}
  \E\left[ \sum_{|z|=1} e^{-V(z)} \right] = 1 \quad \text{and} \quad \E\left[ \sum_{|z|=1} V(z) e^{-V(z)} \right]=0,
\end{equation}
Under mild assumptions, a branching random walk can be reduced to this case by an affine transformation --see Appendix A in~\cite{Jaf9} for a detailed discussion of this question. Furthermore, we assume the following integrability assumptions to hold
\begin{equation}
  \label{eqn:variance}
  \sigma^2 := \E\left[ \sum_{|z|=1} V(z)^2e^{-V(z)}\right]<\infty,
\end{equation}
as well as
\begin{equation}
  \label{eqn:integrability}
  \E\left[ X (\log_+ X)^2\right] + \E\left[ \tilde{X} \log_+ \tilde{X}\right] < \infty,
\end{equation}
where
\begin{equation}
  X := \sum_{|z|=1} e^{-V(z)} \quad \mathrm{and} \quad \tilde{X} := \sum_{|z|=1} V(z)_+e^{-V(z)}.
\end{equation}
Finally, we also assume the point process law $\mathcal{L}$ is non-lattice, i.e.
\begin{equation}
  \label{eqn:nonLattice}
  \forall a,b > 0, \enskip \P\left( \{aV(z)+b, |z|=1\} \subset \mathbb{Z}\right) < 1.
\end{equation}

It follows from \eqref{eqn:boundary} and the branching property of the branching random walk that
\[ W_{n,1} = \sum_{|z|=n} e^{-V(z)} \quad \mathrm{and} \quad Z_n = \sum_{|z|=n} V(z) e^{-V(z)} \]
are martingales. Chen~\cite{Che14++} proved that given \eqref{eqn:variance}, \eqref{eqn:integrability} is necessary and sufficient to obtain the almost sure convergence of $(Z_n)$ toward a non-negative random variable $Z_\infty$. Moreover, we have $S=\{Z_\infty>0\}$ a.s.

Let $\beta > 1$, we write $W_{n,\beta} = \sum_{|z| = n} e^{-\beta V(z)}$. Madaule~\cite{Mad11} proved the convergence in finite-dimensional distributions of $\left(n^{3\beta/2}W_{n,\beta}, \beta >1\right)$. This result in particular implies the convergence in law of the extremal process of the branching random walk toward a proper limiting point process, which is a randomly shifted decorated Poisson point process with exponential intensity.

In this article, we consider a probability measure on the $n^\mathrm{th}$ generation of the branching random walk, defined on the set $\{W_{n,\beta}>0\} = \{ \{|z|=n\} \neq \emptyset\}$ by
\[
  \nu_{n,\beta} = \frac{1}{W_{n,\beta}}\sum_{|z|=n} e^{-\beta V(z)}\delta_z,
\]
which is often called the \emph{Gibbs measure} in the literature. We prove that on the survival event $S$ of the branching random walk, the trajectory followed by an individual chosen according to $\nu_{n,\beta}$ converges, when suitably rescaled, to a Brownian excursion. For a given individual $z \in \T$ such that $|z| \leq n$, we define
\[
  H^{(n)}(z): = \left( \frac{V(z_{\floor{tn}})}{\sqrt{\sigma^2n}} , 0 \leq t \leq \frac{|z|}{n} \right),
\]
the Brownian normalization of the trajectory followed by individual $z$ up to time $|z|$. For all $a<b$, we denote by $\calD([a,b])$ the space of \emph{càdlàg}\footnote{Left-continuous functions with right limits at each point.} functions on $[a,b]$, equipped with the Skorokhod distance. To shorten notation, we will write $\calD([0,1])$ as $\calD$. The function $H^{(n)}(z)$ is an element of $\calD([0,\frac{|z|}{n}])$. For all $\beta > 1$ and $n \in \N$, on the event $\{W_{n,\beta} > 0\}$, we denote the image measure of $\nu_{n,\beta}$ by $H^{(n)}(\cdot)$ by $\mu_{n,\beta}$, i.e. the measure defined on $\calD$ as
\[
  \mu_{n,\beta} = \frac{1}{W_{n,\beta}} \sum_{|z|=n}e^{-\beta V(z)} \delta_{H^{(n)}(z)}.
\]
The following theorem, which is the main result of the article, gives the asymptotic behaviour of the measure $\mu_{n,\beta}$ as $n \to \infty$.
\begin{theorem}
\label{thm:main}
For all $\beta > 1$, conditionally on the survival event $S$ of the branching random walk, we have
\[ \lim_{n \to \infty} \mu_{n,\beta} = \sum_{k \in \N} p_k \delta_{\epsilon^{(k)}} \quad \text{in law}, \]
where $(\epsilon^{(k)})$ is a sequence of i.i.d. normalized Brownian excursions, and $(p_k, k \in \N)$ follows an independent Poisson-Dirichlet\footnote{For a definition of the two-parameters Poisson-Dirichlet distribution, see \cite{PiY97}} distribution with parameter $(\frac{1}{\beta},0)$.
\end{theorem}

An heuristic for this result is developed in the forthcoming Section~\ref{sec:heuristics}.

\begin{remark}
A direct consequence of Theorem~\ref{thm:main} is that the --annealed-- measure $\E(\mu_{n,\beta}|S)$ converges weakly to the law of a normalized Brownian excursion.
\end{remark}

The case of a critical Gibbs measure $\beta=1$ has been investigated in~\cite{Mad12}. It is proved that
\[
  \lim_{n \to \infty} \mu_{n,1}(F) = F(\calM) \textrm{ in probablity},
\]
where $\mathcal{M}$ is a Brownian meander. Formally, in the case $\beta = \infty$, the measure $\mu_{n,\infty}$ is the uniform measure on the trajectories leading to the leftmost position at time $n$, which has been treated in~\cite{Che14+}. For $\beta<1$, Pain \cite{Pain} proved that the trajectory of a particle chosen according to $\nu_{n,\beta}$ behaves as a random walk with positive drift. In particular, after removing the drift, the rescaled trajectory converges toward a Brownian path. Moreover, Pain also obtained the asymptotic behaviour of trajectories sampled according to the law $\mu_{n,\beta_n}$ with $\beta_n \to 1$ as $n \to \infty$, giving a detailed account of the phase transition occurring at $\beta=1$ for this measure.

Using the techniques leading to Theorem~\ref{thm:main}, we obtain informations on the genealogy of two individuals sampled according to the Gibbs measure $\nu_{n,\beta}$. For $z,z'\in \T$, we set $\mathrm{MRCA}(z,z')$ to be the generation at which the most recent common ancestor of $z$ and $z'$ was alive, in other words,
\[ \mathrm{MRCA}(z,z') = \max\{k \leq \min(|z|,|z'|) : z_k = z'_k \}. \]
Derrida and Spohn conjectured in~\cite{DSpo88} that for any $\beta > 1$
\begin{equation}
  \label{conj:ds}
  \nu_{n,\beta}^{\otimes 2}\left( \frac{\mathrm{MRCA}(z,z')}{n} \in dx \right)  \underset{n\to\infty}{\Longrightarrow} \rho_\beta \delta_1 + (1-\rho_\beta) \delta_0,
\end{equation}
where $\rho_\beta$ is a random variable such that $\lim_{\beta \to \infty} \rho_\beta = 1$ and $\lim_{\beta \to 1} \rho_\beta = 0$ in probability. When individuals are sampled according to $\nu_{n,\beta}$ with $\beta < 1$, Chauvin and Rouault~\cite{ChR97} proved that $\mathrm{MRCA}(z,z')$ converges in law, thus $\rho_\beta=0$ for $\beta < 1$. This result was then extended by Pain~\cite{Pain}, who proved the same convergence holds when considering $\nu_{n,\beta_n}$ with $\beta_n \to 1$.

The conjecture of Derrida and Spohn was partially proved by Bovier and Kurkova~\cite{BoK04} for some binary branching processes with Gaussian increments, by Arguin and Zindy~\cite{ArZ12} for the overlapping probability in the $2$-dimensional discrete Gaussian free field, and by Jagannath~\cite{Jag16} for the binary branching random walk. In these articles, the main step of the proofs consist in proving that the model satisfies the so-called \emph{Ghirlanda-Guerra identities}. These identities then imply the convergence in \eqref{conj:ds}, with $\rho_\beta$ being the sum of the squares of the elements of a Poisson-Dirichlet distribution with parameter $(\frac{1}{\beta},0)$. We mention that after the first appearance of this article, the conjecture \eqref{conj:ds} was proved for general branching random walks in \cite{Mal18}, using simple considerations on the limiting extremal process of the branching random walk.

We note here that an immediate consequence of Theorem~\ref{thm:main} is the following result reminiscent of the conjecture of Derrida and Spohn. It characterises the law of first splitting time of trajectories before time $t$.
\begin{corollary}
\label{cor:interest}
For any $\beta > 1$, we have 
\[ \nu_{n,\beta}^{\otimes 2}\left( \inf\{ t > 0 : V(z_\floor{nt}) \neq V(z'_\floor{nt}\} \in dx \right)  \underset{n\to\infty}{\Longrightarrow} \rho_\beta \delta_1 + (1-\rho_\beta) \delta_0,\] 
where $(p_k, k \geq 1)$ has Poisson-Dirichlet distribution with parameter $(\frac{1}{\beta},0)$ and $\rho_\beta= \sum_{k \in \N} p_k^2$.
\end{corollary}

This corollary can be seen as an explicit computation of the well-known fact that in a branching random walk, two individuals within distance $O(1)$ from the leftmost position are either close relatives, or their ancestral lineages split early in the process. In the context of Gibbs measure, the probability of an early splitting is $1-\rho_\beta$. Moreover, we observe that when $\beta\to\infty$, $1-\rho_\beta$  goes to zero. This is consistent with the fact that $\mu_{n,\infty}$ only puts mass on particles which are at the leftmost position at time $n$, which are eventually close relatives when \eqref{eqn:nonLattice} is verified\footnote{If \eqref{eqn:nonLattice} does not hold, then this is no longer true and individuals from distinct families can be at the leftmost position at the same time, cf Pain~\cite[Footnote 3]{Pain}.}. Similarly, when $\beta \to 1$, $1-\rho_\beta$ goes to $1$ ; therefore the corresponding paths are asymptotically independent, which coincides with the weak convergence obtained in~\cite[Equation (3.3)]{Mad12}.

\subsection{Link between the Gibbs measure and the extremal process}
\label{sec:heuristics}

 To prove Theorem~\ref{thm:main}, the main idea is to understand, for all continuous bounded function $F$, the tail decay of the random variable
\begin{equation}
  \label{eqn:defTildeMu}
  \tilde{\mu}_{n,\beta}(F) := n^{3\beta/2} W_{n,\beta}\times\mu_{n,\beta}(F) = n^{3\beta/2} \sum_{|z|=n} e^{-\beta V(z)} F(H^{(n)}(z)),
\end{equation}
with $\tilde{\mu}_{n,\beta}$ the non-normalized version of $\mu_{n,\beta}$, like in Chen~\cite{Che14+} and Madaule~\cite{Mad11}. This tail decay is then used to obtain the limit Laplace transform of $\tilde{\mu}_{n,\beta}(F)$ in Proposition~\ref{propfi}, thanks to the branching property of the branching random walk. We expose here an interpretation of the convergence obtained in Theorem \ref{thm:main}, as well as the main steps of the proof given in this article. We begin with a heuristic for the reason the aforementioned convergence to hold.

We first recall that Madaule \cite{Mad11} proved the convergence of the extremal process of the branching random walk $\sum_{|z|=n} \delta_{V(z) - \frac{3}{2} \log n}$ toward a Cox process with intensity $Z_\infty e^{x}$, decorated by i.i.d. point processes. This convergence can be interpreted as follows. Recall (see e.g. \cite[Theorem 4.5]{Mall}) that particles close to the minimal position at time $n$ are with high probability either close relatives (their MRCA is $n-o(n)$) or come from distinct families (their MRCA is $o(n)$). For  each $n \in \N$, we call a ``local leader'' an individual being the smallest\footnote{Breaking ties uniformly at random.} among all its relatives with a most recent common ancestor alive at a generation larger than $n/2$. The family of positions of local leaders forms a family of (mostly) i.i.d. random variables, which thus converge toward a Poisson point process with intensity $Z_\infty e^x \mathrm{d}x$. However, close relatives of these local leaders are within distance $O(1)$ from their positions. Hence the relative positions of families with respect to their local leader converge toward i.i.d. point processes, that form the decorations of the limiting process. This interpretation was made rigorous in \cite{Mal18}.

We now recall that Chen \cite{Che14+} proved that the rescaled path followed by the individual reaching the minimal position at time $n$ converges toward a Brownian excursion. As the most recent common ancestor between local leaders is of order $1$, each of them can be considered as individual reaching the minimal position at time $n-O(1)$ of an independent branching random walk. Therefore, the trajectory of each local leader converges toward an independent Brownian excursion. The limiting normalized trajectories of all close relatives to a given local leader are the same as the one of the local leader, as they only split from the leader's trajectory in the last few steps, and the normalization will make this difference disappear --this is a reason why we consider uniformly continuous functions in our main theorem.

Therefore, Theorem \ref{thm:main} becomes natural in sight of the following observation : sampling the limiting trajectory of a family at random in a decorated Poisson point process with intensity $Z_\infty e^x$ according to the Gibbs measure with parameter $\beta$ is the same thing as sampling it according to a Poisson-Dirichlet distribution with parameter $(\frac{1}{\beta},0)$ (this can be observed in \cite[Theorem~4.1]{Mal18}).

The main steps of the proof of Theorem \ref{thm:main} follow loosely from the above heuristic. In sight of the many-to-one lemma, which states that additive moments of the branching random walk can be computed through random walk estimates, we first prove in Lemma~\ref{lem:rw} that the rescaled shape of a random walk conditioned to stay non-negative and end up at distance $O(\log n)$ of $0$ is asymptotically independent from the endpoint of that random path. If there were an unique family of individuals close to the minimal position at time $n$, this would allow us to decouple the choice of the individual using the supercritical Gibbs measure (which is concentrated on the individuals close to the minimal position) and the trajectory followed by the family (which converges toward a Brownian excursion).

This decoupling can actually be achieved (in Proposition~\ref{tailprop}) by conditioning the minimal position $m_n$ to be small, say smaller than $\frac{3}{2} \log n - A$ for $A$ large enough. Indeed, in this event of small probability, only one local leader reaches this unusually small minimal position, hence only one family is charged by the Gibbs measure with high probability. Conditioning the minimal displacement to be small is similar to conditioning the random variable $W_{n,\beta}$ to be large, therefore the above observation allows an explicit computation of the tail of the Laplace transform of $\tilde{\mu}_{n,\beta}(F)$.

The proof of Theorem~\ref{thm:main} finally follows from a standard branching argument. We cut the branching random walk at a large but finite generation $k$. Each subtree is then an independent branching random walk, and the ones that contribute to $W_{n,\beta}$ will typically be the ones with a small minimal position at time $n-k$. Therefore, using the estimate of the tail of the Laplace transform allows to give an expression of the Laplace transform. To conclude, it is then enough to compare this Laplace transform with the one of $\sum_{k \in \N} p_k F(\epsilon^{(k)})$.

\paragraph*{Organization of the paper.}
We introduce in Section~\ref{subsec:manytoone} the spinal decomposition and the many-to-one lemma. We obtain in Section~\ref{subsec:rw} some random walk estimates, including the asymptotic independence between the rescaled shape and the limiting position of random walk conditioned to make an excursion. We compute in Section~\ref{sec:taildecay} a tight estimate on the tail of the Laplace transform of $\tilde{\mu}_{n,\beta}(F)$. Finally the proofs of Theorem~\ref{thm:main} and Corollary~\ref{cor:interest} are given in Section~\ref{sec:conclusion}.

\section{Many-to-one lemma and random walk estimates}
\label{sec:usefullemmas}

We introduce in a first time the Lyon's change of measure of the branching random walk and the spinal decomposition. This result enables to compute additive moments of the branching random walk using random walk estimates. In a second part, we consider a centered random walk with finite variance, conditioned to stay above 0 until time $n$ and ending at time $n$ at distance of order $o(\sqrt{n})$, and obtain the asymptotic independence with the rescaled shape and the endpoint of this random walk.

\subsection{Lyon's change of measures and the many-to-one lemma}
\label{subsec:manytoone}

For any $a\in\R$, let $\P_a$ be the probability measure of the branching random walk started from $a$, and let $\E_a$ be the corresponding expectation. We recall that $(W_{n,1}, n \in \N)$ is a non-negative martingale with respect to the natural filtration $\calF_n = \sigma(u,V(u),|u| \leq n)$. We define a new probability measure $\bar{\P}_a$ on $\calF_\infty$ such that for all $n \in \N$,
\begin{equation}
  \label{eqn:measurechange}
  \dfrac{d\bar{\P}_a}{d\P_a}\Big\vert_{\calF_n} = e^a W_{n,1}.
\end{equation}
The so-called spinal decomposition, introduced by Lyons in~\cite{Lyo97} gives an alternative construction of the measure $\bar{\P}_a$, by introduction of a special ray, the ``spine", along which reproduction is modified.

We introduce another point process law $\widehat{\calL}$ with Radon-Nikod\'ym derivative $\sum_\ell e^{-\ell}$ with respect to the law of $\calL$. The branching random walk with spine starts with one individual located at $a$ at time $0$, denoted by $\omega_0$. It generates its children according to the law $\widehat{\calL}$. Individual $\omega_1$ is chosen among the children $z$ of $\omega_0$ with probability proportional to $e^{-V(z)}$. Then, for all $n \geq 1$, individuals at the $n^\mathrm{th}$ generation die, giving birth to children independently according to the law $\calL$, except for the individual $\omega_n$ which uses the law $\widehat{\calL}$. The individual $\omega_{n+1}$ is chosen at random among the children $z$ of $\omega_n$, with probability proportional to $e^{-V(z)}$. We denote by $\T$ the genealogical tree of this process, and by $\hat{\P}_a$ the law of $(V(u),u\in \T, (\omega_n, n \geq 0))$ as we just defined.

\begin{proposition}
\label{prop:spinaldecomposition}
For any $n \in \N$, we have $\left.\hat{\P}_a\right\vert_{\calF_n} = \left.\bar{\P}_a\right\vert_{\calF_n}$. Moreover, for any $z \in \T$ such that $|z|=n$, we have $\hat{\P}_a\left( \omega_n = z \left| \calF_n \right.\right) = \frac{e^{-V(z)}}{W_{n,1}}$, and $(V(\omega_n), n \geq 0)$ is a centered random walk under $\hat{\P}_a$, starting from $a$, and with variance $\sigma^2$.
\end{proposition}

 In particular, this proposition implies the many-to-one lemma, which has been introduced for the first time by Peyri\`ere in~\cite{Pey}, and links additive moments of the branching random walks with random walk estimates.
\begin{lemma}
\label{lem:manytoone}
There exists a centered random walk $(S_n, n \geq 0)$, starting from $a$ under $\P_a$, with variance $\sigma^2$ such that for all $n \geq 1$ and $g : \R^n \to \R_+$ measurable, we have
\begin{equation}
  \label{eqn:manytoone}
  \E_a\left[ \sum_{|z|=n} g(V(z_1),\cdots V(z_n))\right] = \E_a\left[ e^{S_n-a} g(S_1,\cdots S_n) \right]
\end{equation}
\end{lemma}

\begin{proof}
We use Proposition~\ref{prop:spinaldecomposition} to compute
\begin{align*}
  \E_a\left[ \sum_{|z|=n} g(V(z_1),\cdots V(z_n))\right]
  &= \bar{\E}_a \left[ \frac{e^{-a}}{W_{n,1}} \sum_{|z|=n} g(V(z_1), \cdots V(z_n)) \right]\\
  &= e^{-a} \hat{\E}_a\left[ \sum_{|z|=n} \ind{z = \omega_n} e^{V(z)} g(V(z_1),\cdots V(z_n)) \right]\\
  &= \hat{\E}_a \left[ e^{V(\omega_n) - a} g(V(\omega_1), \cdots, V(\omega_n)) \right].
\end{align*}
Therefore we define the random walk $S$ under $\P_a$ to have the same law as $(V(\omega_n), n \geq 0)$ under $\hat{\P}_a$, which ends the proof.
\end{proof}

\subsection{Approximation of a random walk excursion}
\label{subsec:rw}

In this subsection, $(S_n, n \geq 0)$ is a centered random walk on $\R$ with finite variance $\sigma^2$, which is non-lattice, i.e. the support of the law of $S_1$ is not included in some discrete additive subgroup of $\R$. We write, for $0 \leq m \leq n$, $\mS_{[m,n]} = \min_{m \leq k \leq n} S_k$ and $\mS_n = \mS_{[0,n]}$ the minimal position of the random walk until time $n$. We introduce in a first time a piece of notation, before computing the probability for a random walk to make an excursion of length $n$ above 0, and the asymptotic independence between the endpoint and the shape of the excursion, on that event.

We denote by $\calC_b(\calD)$ be the set of continuous bounded functions from $\calD$ to $\R$, as well as $\calC^u_b(\calD)\subset\calC_b(\calD)$ the collection of uniformly continuous functions. In this section, we often prove the estimates for uniformly continuous functions in a first time, before extending them to any continuous bounded functions.

\subsubsection{Some random walk notation and preliminary results}

\paragraph*{The ballot theorem.}
We present the following estimates, which bound the probability for a random walk to make an excursion of length $n$ above a given level. Let $\lambda\in(0,1)$. There exists a constant $c_1>0$ such that for any $b\geq a\geq 0$, $x,y\geq 0$ and $n\geq 1$, we have
\begin{equation}
\label{eq:elementary}
  \P_x\Big(S_n\in[y+a,y+b], \mS_n\geq0, \mS_{[\lambda n,n]}\geq y\Big)\leq c_1(1+x)(1+b-a)(1+b) n^{-3/2}.
\end{equation}
This classical estimate can be found for example in~\cite[Lemma 2.4]{AiS10}.

\paragraph*{Ladder epochs and height processes.}
We denote by $(\tau_k^+, k \geq 0)$ and $(H_k^+, k \geq 0)$ the strictly ascending ladder epochs and the height process, writing $\tau^+_0=0, H^+_0=0$ and, for $k \geq 1$,
\begin{equation}
  \label{eqn:ladderepoch_def}
  \tau_k^+ = \inf\{n> \tau_{k-1}^+ : S_n > H_{k-1}^+\} \quad \mathrm{and} \quad H_k^+ = S_{\tau_k^+}.
\end{equation}
Note that $(\tau_k^{+}, k\geq0)$ and $(H_k^{+}, k\geq0)$ are renewal processes, i.e., random walks with i.i.d. non-negative increments. Similarly, we write $\tau^-$ and $H^-$ the strictly ascending ladder epoch and the height process associated to $-S$.
It is given in \cite[Theorem A]{Koz76} that there exist two constants $C_{\pm}>0$ such that
\begin{equation}
  \P(\tau_1^\pm>n) = \P(\min_{k \leq n}(\mp S_k) \geq 0) =\frac{C_\pm}{\sqrt{n}}+o(n^{-1/2}).
  \label{eq:probstaypositive}
\end{equation}

\paragraph{Renewal function.}
We write $V^-(\cdot)$ (respectively $V^+(\cdot)$) the renewal function associated to $(H^-_k,k \geq 0)$ (resp. $(H^+_k, k \geq 0)$), defined by
\begin{equation}
  \label{eqn:renwal_def}
  \forall x \geq 0, \quad V^-(x) = \sum_{k \geq 0} \P\left(H_k^- \leq x\right).
\end{equation}
Observe that $V^-$ is a non-decreasing, right-continuous function with $V^-(0)=1$. We can rewrite $V^-$ in the following way
\begin{equation}
  V^-(x) = \sum_{k \geq 0} \P\left( -x \leq S_k < \mS_{k-1} \right)
\end{equation}
As a consequence of the Renewal Theorem in~\cite[p. 360]{Fel71}, there exist two constants $c_\pm>0$ such that
\begin{equation}
  V^\pm(x) \sim_{x \to \infty} c_\pm x.
\end{equation}

\paragraph*{Local measure of the random walk staying non-negative.}
We introduce, for $n \geq 1$, the measure
\begin{equation}
  \pi_n^+(x, dy):=\P_x\Big(\mS_n\geq0, S_n\in dy\Big),
\end{equation}
Let $K>0$, it has been proved by Doney~\cite{Don12} that uniformly in $x=o(\sqrt{n})$ and $y=o(\sqrt{n})$,
\begin{equation}
  \label{eq:doneyo}
  \pi_n^+(x,[y,y+K]) = \frac{1}{\sigma \sqrt{2\pi } n^{3/2}}V^-(x)\int_{[y, y+K]}V^+(z)dz\left(1+o_n(1)\right).
\end{equation}
and that uniformly in $x=o(\sqrt{n})$ and $y\in[0,\infty)$,
\begin{equation}
  \label{eq:doney+}
  \pi_n^+(x, [y, y+K])=\frac{C_-y}{\sigma^2 n^{3/2}} e^{-\frac{y^2}{2n\sigma^2}}KV^-(x) + o(n^{-1}).
\end{equation}
Obviously, similar estimates hold for $\pi^-$ the measure associated to $-S$.

\paragraph*{Random walk conditioned to stay non-negative.}
We observe that the renewal function $V^-$ is invariant for the semigroup of the random walk killed when it first enters the negative half-line $(-\infty,0)$, i.e.
\begin{equation}
  \label{eq:invariant}
  \forall x \geq 0,\ \forall N \in \N,\ V^-(x) = \E\left[ V^-(x+S_N) \ind{\mS_N \geq -x} \right].
\end{equation}
This equality can be found in~\cite{Koz76}.

Using \eqref{eq:invariant}, for all $x \geq 0$, we define the probability measure $\P_x^{\uparrow}$ by
\begin{equation}\label{eq:Pup}
  \P_x^{\uparrow}(B):=\frac{1}{V^-(x)}\E_x\left(1_B V^-(S_N); \mS_N\geq0\right),
\end{equation}
for $N \geq 1$ and $B \in \sigma(S_0,\ldots S_N)$. We call $\P_x^{\uparrow}$ the law of the random walk conditioned to stay positive. For any positive sequence $(x_n)$ such that $\frac{x_n}{\sqrt{\sigma^2n}}\rightarrow x\geq0$, we have the following invariance principle, proved in~\cite[Theorem~1.1]{CCh08},
\begin{equation}\label{eq:cvtobessel}
\forall F\in\calC_b(\calD),\quad \E_{x_n}^{\uparrow}\Big(F\left(\tfrac{S_{\floor{nt}}}{\sqrt{\sigma^2n}}; t\in[0,1]\right)\Big)\xrightarrow[n\rightarrow\infty]{} \E_x\Big(F(R)\Big),
\end{equation}
where $R=(R(t); t\geq0)$ is a three-dimensional Bessel process.

We also state another functional central limit theorem related to \eqref{eq:cvtobessel}, which has been proved by Iglehart~\cite{Igl74}, Bolthausen~\cite{Bol76} and Doney~\cite{Don85}.
\begin{equation}\label{eq:cvtomeander}
\forall F\in\calC_b(\calD),\quad \E\Big(F\left(\tfrac{S_{\floor{nt}}}{\sqrt{\sigma^2n}}; t\in[0,1]\right)\Big\vert \mS_n\geq0\Big)\xrightarrow[n\rightarrow\infty]{} \E\Big(F(\calM)\Big),
\end{equation}
where $\calM=(\calM(t); t\in[0,1])$ is a Brownian meander process. The following equality from Imhof~\cite{Imh84} reveals the relation between these two limit processes. For any $t\in(0,1]$, 
\begin{equation}\label{eq:meandertobessel}
\forall \Phi\in\calC_b(\calD[0,t]), \quad\E\left[ \Phi(\calM(u), u \leq t)\right] = \sqrt{\frac{\pi}{2}} \E\left[ \frac{\sqrt{t}}{R(t)}\Phi(R(u), u \leq t) \right].
\end{equation}

\paragraph{Decomposition of the excursions.}
We write $\rho^t_{x,y}=(\rho^t_{x,y}(s), s\in[0,t])$ for a 3-dimensional Bessel bridge of length $t \in [0,1]$ between $x$ and $y$, where $x,y \in \R_+$. Intuitively, 
\begin{equation}
  \forall F\in\calC_b(\calD),\quad\E\Big(F(\rho^t_{x,y})\Big)=\E_x\Big(F(R(s), s\in[0,t])\Big\vert R(t)=y\Big).
\end{equation}
For all $\lambda \in (0,1)$, $G_1 \in \calC(\calD([0,\lambda]))$, $G_2 \in \calC(\calD([0,1-\lambda]))$ and $x \in \calD$ we set
\[
  G_1 \star G_2(x) = G_1(x_s, s \leq \lambda) G_2(x_{s+\lambda},s \leq 1 - \lambda). 
\]

\begin{lemma}
\label{lem:gamma}
Let $\epsilon = (\epsilon_t, t \in [0,1])$ be a normalized Brownian excursion. We write $(\mathcal{M}_t, t \in [0,\lambda])$ and $(\rho^{1-\lambda}_{x,0}(t), x \in \R^+,t \in [0,1-\lambda]$ two independent processes, with $\mathcal{M}$ a Brownian meander of length $\lambda$ and $\rho^{1-\lambda}_{x,0}$ a Bessel bridge between $x$ and $0$ of length $1-\lambda$. We have
\begin{equation}\label{eq:gamma}
 \E\left[ G_1 \star G_2(\epsilon)\right]=\sqrt{\frac{2}{\pi}} 
 \frac{1}{\lambda^{1/2}(1-\lambda)^{3/2}} \E\left[ \mathcal{M}(\lambda) e^{-\frac{\mathcal{M}(\lambda)^2}{2(1-\lambda)}} G_1\left(\mathcal{M}\right) G_2\left(\rho^{1-\lambda}_{\mathcal{M}(\lambda),0}\right)\right].
\end{equation}
\end{lemma}

\begin{proof}
We show that both sides in \eqref{eq:gamma} are equal to
\begin{equation}
  \sqrt{\frac{2}{\pi}}\int_0^\infty \frac{x^2}{\lambda^{3/2}(1-\lambda)^{3/2}}e^{-\frac{x^2}{2\lambda(1-\lambda)}} \E\left[G_1\left(\rho^\lambda_{0,x}\right)\right]\E\left[G_2\left(\rho^{1-\lambda}_{x,0}\right)\right]dx.
\end{equation}

Recall that $\epsilon$ has the same law as $\rho^1_{0,0}$ a 3-dimensional Bessel bridge of length 1. Conditioning on the value of $\rho^1_{0,0}(\lambda)$, we have
\begin{multline*}
  \E\left[ G_1 \star G_2(\epsilon) \right]
  = \E\left[ G_1 \left(\rho^1_{0,0}(s), s \leq \lambda\right) G_2\left( \rho^1_{0,0}(s), \lambda \leq s \leq 1 \right) \right]\\
  = \int_0^\infty \P\left(\rho^1_{0,0}(\lambda)\in dx\right)\E\left[ G_1 \left(\rho^1_{0,0}(s), s \leq \lambda\right) G_2\left( \rho^1_{0,0}(s), \lambda \leq s \leq 1 \right) \Big\vert \rho^1_{0,0}(\lambda)=x\right],
\end{multline*}
where the density of $\rho^1_{0,0}(\lambda)$ is $\P\left(\rho^1_{0,0}(\lambda)\in dx\right)=\sqrt{\frac{2}{\pi}}\frac{1}{\lambda^{3/2} (1-\lambda)^{3/2}}x^2e^{-\frac{x^2}{2\lambda(1-\lambda)}}1_{x\geq0}dx$. It hence follows that 
\begin{align*}
 & \E\left[ G_1 \star G_2(\epsilon) \right]\\
=&  \sqrt{\frac{2}{\pi}} \int_0^{\infty} dx\frac{x^2}{\lambda^{3/2}(1-\lambda)^{3/2}} e^{- \frac{x^2}{2\lambda(1-\lambda)}} \E\left[ G_1 \left(\rho^1_{0,0}(s), s \leq \lambda\right) G_2\left( \rho^1_{0,0}(s), \lambda \leq s \leq 1 \right) \Big\vert \rho^1_{0,0}(\lambda)=x\right].
\end{align*}
Applying the Markov property at time $\lambda$ yields
\begin{align*}
 &\E\left[ G_1 \left(\rho^1_{0,0}(s), s \leq \lambda\right) G_2\left( \rho^1_{0,0}(s), \lambda \leq s \leq 1 \right) \Big\vert \rho^1_{0,0}(\lambda)=x\right]\\
 &= \E\left[ G_1 \left(\rho^1_{0,0}(s), s \leq \lambda\right)\Big\vert \rho^1_{0,0}(\lambda)=x\right]\E\left[G_2\left(\rho^{1-\lambda}_{x,0}\right)\right]\\
 &= \E\left[G_1\left(\rho^\lambda_{0,x}\right)\right] \E\left[G_2\left(\rho^{1-\lambda}_{x,0}\right)\right].
\end{align*}
As a consequence
\begin{equation*}
  \E\left[ G_1 \star G_2(\epsilon)\right]
  =\sqrt{\frac{2}{\pi}}\int_0^\infty \frac{x^2}{\lambda^{3/2} (1-\lambda)^{3/2}}e^{-\frac{x^2}{2\lambda(1-\lambda)}}\E\left[G_1\left(\rho^\lambda_{0,x}\right)\right] \E\left[G_2\left(\rho^{1-\lambda}_{x,0}\right)\right]dx.
\end{equation*}

On the other hand, writing
\[ \Gamma(G_1,G_2,\lambda) = \sqrt{\frac{2}{\pi}} \frac{1}{\lambda^{1/2}(1-\lambda)^{3/2}} \E\left[ \mathcal{M}(\lambda) e^{-\frac{\mathcal{M}(\lambda)^2}{2(1-\lambda)}} G_1\left(\mathcal{M}\right) G_2\left(\rho^{1-\lambda}_{\mathcal{M}(\lambda),0}\right)\right], \]
by \eqref{eq:meandertobessel} we have
\begin{align*}
  \Gamma(G_1, G_2,\lambda)
  &= \sqrt{\frac{2}{\pi}}\frac{1}{\lambda^{1/2}(1-\lambda)^{3/2}} \E\left[ G_1\left(\mathcal{M}\right) 
    \mathcal{M}(\lambda) e^{-\frac{\mathcal{M}(\lambda)^2}{2(1-\lambda)}} 
    G_2\left(\rho^{1-\lambda}_{\mathcal{M}(\lambda),0}\right)\right]\\
  &= \frac{1}{(1-\lambda)^{3/2}}  \E\left[G_1\left( R(s); s\in[0,\lambda]\right) e^{-\frac{R(\lambda)^2}{2(1-\lambda)}} G_2\left(\rho^{1-\lambda}_{R(\lambda),0}\right)\right],
\end{align*}
where $(R(s);0\leq s\leq \lambda)$ is a Bessel process independent with $(\rho^{1-\lambda}_{x,0})$. We now condition on the value of $R(\lambda)$ --recall that the law of $R(\lambda)$ is $\P(R(\lambda)\in dx)= \sqrt{\frac{2}{\pi \lambda^3}}x^2 e^{-x^2/(2\lambda)}1_{x\geq0}dx$-- to obtain
\begin{align*}
  &\Gamma(G_1, G_2,\lambda)\\
  =&\frac{1}{(1-\lambda)^{3/2}} \int_0^\infty \P\Big(R(\lambda)\in dx\Big) \E\left[G_1\left( R(s); s\in[0,\lambda]\right) e^{-\frac{R(\lambda)^2}{2(1-\lambda)}} \E\left[G_2\left(\rho^{1-\lambda}_{R(\lambda),0}\right)\right]\Big\vert R(\lambda)=x\right]\\
  =&\sqrt{\frac{2}{\pi}}\int_0^\infty \frac{x^2}{\lambda^{3/2} (1-\lambda)^{3/2}}e^{-\frac{x^2}{2\lambda(1-\lambda)}}\E\left[G_1\left(\rho^\lambda_{0,x}\right)\right] \E\left[G_2\left(\rho^{1-\lambda}_{x,0}\right)\right]dx.
\end{align*}
We conclude that $
 \E\left[ G_1 \star G_2(\epsilon)\right] = \Gamma(G_1, G_2,\lambda)$.
\end{proof}

\subsubsection{Asymptotic independence of the endpoint and the shape of the trajectory in a random walk excursion}

For $n \in \N$, let $\ns^{(n)}$ be the normalized path of the random walk $S$, defined, for $t \in [0,1]$ by
\begin{equation}
  \ns^{(n)}_t:=\frac{S_\floor{nt}}{\sqrt{\sigma^2 n}},
\end{equation}
also written $\ns$ when the value of $n$ is unambiguous. Clearly, $(\ns^{(n)}_t, t\in[0,1])\in \calD$. We prove in this section that conditionally on $\{\mS_{n} \geq 0\}$ and $\{S_n = o(\sqrt{n})\}$, the endpoint $S_n$ is asymptotically independent of rescaled shape $\ns$ of the excursion.

We begin with the following estimate, for a random walk starting at time $0$ within distance~$\sqrt{n}$ from the boundary.
\begin{lemma}
\label{lem:besselpont}
Let $(y_n)_{n \geq 1}$ be a non-negative sequence such that $\lim_{n \to \infty} \frac{y_n}{n^{1/2}}=0$. There exists $C_\star =\frac{C_+}{\sigma}$ (with $C_+$ the constant defined in \eqref{eq:probstaypositive}) such that for all $K \in \R^+$ and $F\in\calC_b(\calD)$, we have
\begin{multline}
  \lim_{n \to \infty} \sup_{x \in \r+} \bigg| n\E_{x \sigma \sqrt{n}} \left[ F(\ns^{(n)}) ; \mS_n \geq y_n, S_n \in [y_n,y_n+K] \right] \\ - C_\star g(x) \E\left(F(\rho^1_{x,0})\right) \int_0^K V^+(z) dz\bigg| = 0,
\end{multline}
where $g : x \mapsto xe^{-\frac{x^2}{2}}\ind{x \geq 0}$ and $\rho^1_{x,0}$ is a three-dimensional Bessel bridge of length 1 from $x$ to $0$.
\end{lemma}

\begin{proof}
The proof of this lemma is largely inspired by the arguments in~\cite{CCh13}. By to Lemma~A.1 of \cite{Pain}, it is enough to prove this convergence for any $F\in\calC_b^u(\calD)$.

Let $n \in \N$ and $F$ uniformly continuous, we have, in terms  of the local measure
\begin{multline}\label{eq:besselbridgecond}
  \E_{x\sigma\sqrt{n}} \left( F(\ns) ; \mS_n \geq y_n, S_n \in [y_n,y_n+K] \right)\\
  =\left\{ \E_{x\sigma \sqrt{n} -y_n} \left( \left. F(\ns)\right| \mS_n \geq 0, S_n \in [0,K] \right) + o_n(1) \right\} \pi_n^+ (x\sigma \sqrt{n}-y_n, [0,K]).
\end{multline}

Recall that \eqref{eq:doneyo} and \eqref{eq:doney+} give estimates of $\pi_n^+\left( x, [y,y+K]\right)$ when $x=o(\sqrt{n})$. We first extend this result by showing there exists $C_\star>0$ such that uniformly in $x \in [0,\infty)$, 
\begin{equation}
  \label{eq:doneyO}
  \pi_n^+\left( x, [0,K]\right) = \frac{C_\star}{n} \int_0^K V^+(z)dz g\left( \frac{x}{\sigma \sqrt{n}}\right) + o\Big(\frac{1}{n}\Big).
\end{equation}
Let $n \in \N$, we write $S^-_k = S_{n-k}-S_n$ for $0\leq k\leq n$, the ``time-reversal random walk'', which has the same law as $-S$. We observe that
\begin{align*}
  \pi_n^+\left(x, [0,K]\right)
  =& \P_{x}\left(\mS_n \geq 0, S_n \in [0,K]\right)
  = \P\left(\min_{0\leq k\leq n} S^-_k\geq S^-_n-x \geq-K\right)\\
  =&\sum_{j=0}^n\P\left(T=j, \min_{0\leq k\leq n} S^-_k\geq S^-_n-x\geq-K\right),
\end{align*}
where $T:=\min\{j\leq n: S_j^-=\min_{0\leq k\leq n}S^-_k\}$. Applying the Markov property at time $T=j$ yields that
\begin{equation}
  \label{eq:sum}
  \pi_n^+\left(x, [0,K]\right)=\sum_{j=0}^n \E\left(\ind{-K\leq S_j^-<\min_{0\leq k\leq j-1}S^-_k}\pi_{n-j}^-(0, [x-K-S_j^-, x])\right) = \sum_{j=0}^n \tau_j,
\end{equation}
where $\tau_j:=\E\left(\ind{-K\leq S_j^-<\min_{0\leq k\leq j-1}S^-_k}\pi_{n-j}^-(0, [x-K-S_j^-, x])\right)$ for all $0 \leq j \leq n$. Applying \eqref{eq:doney+} to $\pi^-$, uniformly in $y \geq 0$, we have
\begin{equation}\label{eq:upperF-}
  \pi_{n}^-(0, [y, y+K])=\frac{C_+}{\sigma n} \left(g(\tfrac{y}{\sigma\sqrt{n}})K+o_n(1)\right).
\end{equation}
Therefore,
\begin{align}
  \sum_{0\leq j\leq\sqrt{n}} \tau_j &= \frac{C_+}{\sigma n} \left(g(\tfrac{x}{\sigma\sqrt{n}})+o_n(1)\right) \sum_{0\leq j\leq\sqrt{n}} \E\left[\ind{-K\leq S_j^-<\min_{0\leq k\leq j-1}S^-_k} (K+S_j^-)\right] \nonumber\\
  &= \frac{C_+}{\sigma n}\left( g(\tfrac{x}{\sigma \sqrt{n}}) + o_n(1)\right) \left(\int_0^K V^+(z)dz + o_n(1)\right). \label{eq:F+O}
\end{align}
Using \eqref{eq:upperF-}, we observe there exists $c_2>0$ such that for all $n \in \N$ and $y \geq 0$, $\pi_n^-(0, [y, y+K])\leq \frac{c_2(1+K)}{n+1}$, which implies
\begin{align*}
  \sum_{\sqrt{n}<j\leq n} \tau_j
  &\leq \sum_{\sqrt{n}<j\leq n} \frac{c_2}{n-j+1} \E\left[\ind{-K\leq S_j^-<\min_{0\leq k\leq j-1}S^-_k}(1+K+S_j^-)\right]\\
  &\leq \sum_{\sqrt{n}<j\leq n} \frac{c_2(1+K)}{n-j+1}\P(\mS_j\geq0, S_j\leq K)&&\text{by time-reversal}\\
  &\leq \sum_{\sqrt{n}<j\leq n} \frac{c_3(1+K)^2}{(n-j+1)j^{3/2}} &&\text{using \eqref{eq:elementary}},
\end{align*}
so $\sum_{\sqrt{n}<j\leq n} \tau_j = o(n^{-1})$. As a consequence, for $C_\star = \frac{C_+}{\sigma}$, uniformly in $x \geq 0$, \eqref{eq:sum} becomes
\begin{equation}
  \pi_n^+\left(x, [0,K]\right) = \frac{C_\star}{n}g(\tfrac{x}{\sigma \sqrt{n}})\int_0^K V^+(z)dz + o(n^{-1}).
\end{equation}

Plugging this result into \eqref{eq:besselbridgecond}, we obtain that, uniformly in $x \geq 0$
\begin{multline*}
  \E_{x\sigma\sqrt{n}}\left(F(\ns); \mS_n\geq y_n, S_n\in[y_n, y_n+K]\right)\\
  = \frac{C_\star g(x)}{n} \int_0^K V^+(z) dz \E_{(x\sigma\sqrt{n}-y_n)}\left[F(\ns)\left\vert \mS_n\geq 0, S_n\in[0, K]\right. \right]+o(n^{-1}).
\end{multline*}
As $\lim_{x \to \infty} g(x)= 0$, it remains to prove that for any $K_0>0$ fixed,
\begin{equation}
  \label{eq:bridge}
  \lim_{n \to \infty} \sup_{x \in [0,K_0]}\left|\E_{(x\sigma\sqrt{n}-y_n)}\left[F(\ns)\left\vert \mS_n\geq 0, S_n\in[0, K]\right. \right] - \E(F(\rho^{x,0})) \right| = 0.
\end{equation}

Let $K_0>0$ and $\varepsilon>0$, we prove that \eqref{eq:bridge} holds for any $F\in \mathcal{C}_b(\mathcal{D}([0,1-\varepsilon]))$.

Let $M:=\floor{(1-\varepsilon) n}$. For any $x\geq0$, applying the Markov property at time $M$ gives
\begin{align*}
  \E_x\left(F(\ns)\left\vert \mS_n\geq0, S_n\in[0, K]\right.\right)
  &= \frac{\E_x\left[F(\ns);\mS_n\geq0, S_n\in[0, K]\right]}{\P_x\left[\mS_n\geq0, S_n\in[0, K]\right]}\\
  &=\frac{\E_x\left[F(\ns)\ind{\mS_M\geq0} \P_{S_{M}}\left(\mS_{n-M}\geq0, S_{n-M}\in[0, K]\right)\right]}{\pi_n^+(x, [0,K])}\\
  &=\frac{\E_x\left[F(\ns)\ind{\mS_M\geq0}  \pi_{n-M}^+(S_M, [0,K])\right]}{\pi_n^+(x, [0,K])}.
\end{align*}
We set $x_n:=x\sigma\sqrt{n}-y_n$. By change of measures introduced in \eqref{eq:Pup}, we observe that
\begin{align*}
  \E_{x_n} \left[ F(\ns) \left\vert \mS_n\geq 0, S_n\in [0, K]\right.\right]
  =&\frac{\E^{\uparrow}_{x_n}\left[F(\ns) \frac{V^-(x_n)}{V^-(S_{M})} \pi_{n-M}^+(S_M, [0,K])\right]}{\pi_n^+(x_n, [0,K])}\\
  =& \E^{\uparrow}_{x_n} \left[F(\ns) f_{\varepsilon,x_n}^{n}(\tfrac{S_M}{\sigma\sqrt{n}})\right],
\end{align*}
where we write (recalling that $M=\floor{n(1-\varepsilon)}$)
\begin{equation}
  f_{\varepsilon,x_n}^{n}(z):=\frac{V^-(x_n)}{\pi_n^+(x_n,[0,K])}\frac{\pi_{n-M}^+(z\sigma\sqrt{n}, [0,K])}{V^-(z\sigma\sqrt{n})}.
\end{equation}

On the other hand, for a Bessel bridge $\rho^1_{x,0}$, by the Markov property at time $1-\varepsilon$,
\begin{equation}
  \E\left(F(\rho^1_{x,0})\right)=\E_x\left[F\Big(R(s); s\in[0,1-\varepsilon]\Big)f_{\varepsilon,x}\Big(R(1-\varepsilon)\Big)\right],
\end{equation}
where
\begin{equation}
  f_{\varepsilon,x}(z):=\frac{e^{-z^2/(2\varepsilon)}}{\varepsilon^{3/2}e^{-x^2/2}}.
\end{equation}
As a result,
\begin{multline*}
  \left|\E_{x_n}\left[F(\ns)\left\vert\mS_n\geq 0, S_n\in[0, K]\right.\right)- \E\left(F(\rho^1_{x,0})\right)\right|\\
  \leq \left|\E^{\uparrow}_{x_n}\left[ F(\ns) f_{\varepsilon,x_n}^{n}(\tfrac{S_M}{\sigma\sqrt{n}})\right]-\E^{\uparrow}_{x_n}\left(F(\ns) f_{\varepsilon,x}(\tfrac{S_M}{\sigma\sqrt{n}})\right)\right|\\+\left|\E^{\uparrow}_{x_n}\left(F(\ns) f_{\varepsilon,x}(\tfrac{S_M}{\sigma\sqrt{n}})\right)-\E_x\left[F\Big(R(s); s\in[0,1-\varepsilon]\Big)f_{\varepsilon,x}\Big(R(1-\varepsilon)\Big)\right]\right|,
\end{multline*}
which leads to
\begin{multline*}
 \left|\E_{x_n}\left[F(\ns)\left\vert\mS_n\geq 0, S_n\in[0, K]\right.\right)-\E\left(F(\rho^1_{x,0})\right)\right| \\
 \leq \sup_{z\geq 0 , x \in [0,K]}\left| f_{\varepsilon,x_n}^{n}(z)-f_{\varepsilon,x}(z)\right|||F||_\infty\\
 +\left|\E^{\uparrow}_{x_n}\left(F(\ns) f_{\varepsilon,x}(\tfrac{S_M}{\sigma\sqrt{n}})\right)-\E_x\left[F\left(R(s); s\in[0,1-\varepsilon]\right) f_{\varepsilon,x}\left(R(1-\varepsilon)\right)\right]\right|.
\end{multline*}
By use of \eqref{eq:doneyO} and \eqref{eq:doneyo}, we have
\begin{equation}
\lim_{n\rightarrow\infty} \sup_{z\geq 0 , x \in [0,K]}\left| f_{\varepsilon,x_n}^{n}(z)-f_{\varepsilon,x}(z)\right|=0.
\end{equation}
It follows from \eqref{eq:cvtobessel} that
\begin{equation}
  \lim_{n \to \infty} \sup_{x \in [0,K]} \left|\E^{\uparrow}_{x_n}\left(F(\ns) f_{\varepsilon,x}(\frac{S_M}{\sigma\sqrt{n}})\right)-\E_x\left[F\Big(R(s); s\in[0,1-\varepsilon]\Big)f_{\varepsilon,x}\Big(R(1-\varepsilon)\Big)\right]\right| = 0.
\end{equation}

As a result, to complete the proof of Lemma~\ref{lem:besselpont}, it is enough to check the tightness of $\ns$ under $\P_{x_n}\left(\cdot\left| \mS_n\geq0, S_n\in[0, K]\right.\right)$, as the previous equation gives convergence in finite dimensional distributions of this quantity. By Theorem 15.3 of~\cite{Bil99}, for any $\eta>0$, it suffices to say that
\begin{equation}
\lim_{\delta\to 0}\lim_{n\to \infty} \sup_{x \in [0,K]}\P_{x_n}\left(\left.\sup_{0\leq k\leq \delta n}S_{n-k}\geq \eta\sigma\sqrt{n}\right| \mS_n\geq0, S_n\in[0, K]\right)=0,
\end{equation}
which holds immediately by time reversal properties.
\end{proof}

Using Lemma~\ref{lem:besselpont}, we obtain the main result of this section, the joint convergence of this normalized path and the terminal position in a random walk excursion.
\begin{lemma}
\label{lem:rw}
Let $f : \R_+ \to \R$ be a Riemann-integrable function such that there exists a non-increasing function $\widehat{f}$ verifying $|f(x)| \leq \widehat{f}(x)$ and $\int_{\R_+} x \widehat{f}(x) dx < \infty$. Let  $(r_n)$ be a non-negative sequence such that $\limsup_{n \to \infty} \frac{r_n}{\log n}<\infty$. There exists a constant $C_1>0$ such that for all such functions $f$, $\lambda \in (0,1)$ and $F\in\calC_b(\calD)$,
\begin{equation}
  \label{eq:convofrw}
  \lim_{n \to \infty} \sup_{y \in [0,r_n]} \left| n^{3/2} \E\left[ F\left(\ns\right) f(S_n-y) ; \mS_n \geq 0, \mS_{[\lambda n, n]}  \geq y \right] - C_1\E\left[ F(\epsilon) \right] \int_{\r_+} f(x) V^+(x) dx\right| = 0,
\end{equation}
where $\epsilon=(\epsilon_t, t \in [0,1])$ is a standard Brownian excursion.
\end{lemma}

\begin{proof}
This lemma is a slight refinement of Lemma 2.4 in~\cite{Che14+}, which proved the above convergence for any function $F$ that depends only on the values of $\ns$ on the interval $[0,\alpha]$ for some $\alpha<1$. Without loss of generality, we assume $0 \leq F \leq 1$. Similarly, up to a decomposition of $f$ in its positive and negative part, we can assume without loss of generality that $f$ is non-negative.
 For convenience, we set
\begin{equation}
  \chi(F,f) := \E\left[ F(\ns) f(S_n-y) ; \mS_n \geq 0, \mS_{[\lambda n,n]} \geq y \right].
\end{equation}

For any $K>0$, writing $f_K(x) = f(x) \ind{x \in [0,K]}$, we observe that
\begin{equation*}
  \chi(F,f)=\chi\left(F, f_K \right)+\chi\left(F,f-f_K\right).
\end{equation*}
As $0\leq F\leq 1$, we have $\chi\left(F,f-f_K\right) \leq  \chi\left(1,f-f_K\right)$ (using the fact that $f \geq f_K$), and
\begin{align*}
  \chi\left(1,f-f_K\right)
  &\leq \sum_{j=K}^{\infty} \E\left(f(S_n-y); \underline{S}_n\geq 0, \underline{S}_{[\lambda n, n]}\geq y, S_n\in[y+j, y+j+1]\right)\\
  &\leq \sum_{j=K}^{\infty} \widehat{f}(j) \underbrace{\P\left( \underline{S}_n \geq 0, \underline{S}_{[\lambda n,n]} \geq y, S_n \in [y+j,y+j+1]\right)}_{\leq c_1 (2+j)n^{-3/2}},
\end{align*}
by use of \eqref{eq:elementary}. As $\int_0^{\infty} x \widehat{f}(x)dx < \infty$, we have $\lim_{K \to \infty} \sum_{j=K}^{\infty} (2+j) \widehat{f}(j) = 0$.

Therefore, we only need to estimate $\chi(F,f_K)$, and, as $f$ is Riemann-integrable, it is enough to consider functions of the form $\ind{x\in[0,K]}$, for $K \in \R$. We now compute an equivalent of
\[ \chi(F,K) := \chi(F,\ind{[0,K]}) = \E\left[ F(\ns) ; \mS_n \geq 0, \mS_{[\lambda n,n]} \geq y, S_n \leq y+K\right]. \]

Using Lemma A.2 of \cite{Pain}, we only need to prove the uniform convergence for $F = G_1 \star G_2$ where $G_1 \in \calC^u_b(\calD([0,\lambda]))$ and $G_2 \in \calC^u_b(\calD([0,1-\lambda]))$ are two uniformly continuous bounded functions. We prove that uniformly in $y \in [0,r_n]$, we have
\begin{equation}
\label{eq:convofrwg+g}
 \lim_{n \to \infty} \left| n^{3/2} \chi(G_1 \star G_2(\ns), K)- C_1\E\left[ G_1 \star G_2(\epsilon) \right] \int_0^K V^+(x) dx\right| = 0.
\end{equation}
Applying the Markov property at time $m=m_n:=\lfloor \lambda n\rfloor$, we have
\begin{equation}
  \label{BRWmarkovprop}
  \chi\left(G_1\star G_2, K\right)
  =\E\left[G_1\left(\ns_t; t\in[0,\lambda]\right)\Psi_{K,G_2}\left(\tfrac{S_{m}}{\sigma \sqrt{n}}\right); \underline{S}_{m}\geq 0\right],
\end{equation}
where for $x\geq0$,
\[
   \Psi_{K,G_2}(x):=\E_{x\sigma\sqrt{n}}\left[G_2\left(\tfrac{S_{\floor{n(t+\lambda)}-m}}{\sigma\sqrt{n}}; t\in [0,1-\lambda]\right); S_{n-m}\leq y+K, \underline{S}_{n-m}\geq y\right].
\]
Recall that $\rho^t_{x,y}$ is a 3-dimensional Bessel bridge of length $t$ between $x$ and $y$. Using Lemma~\ref{lem:besselpont}, uniformly in $x\geq0$ and $y\in[0,r_n]$, we have
\[
  \Psi_{K,G_2}(x)=\frac{C_\star}{ (1-\lambda)n}\int_0^K V^+(z)dz \psi(x)+o(n^{-1}),
\]
where $\psi(x):=g\left(\tfrac{x}{\sqrt{1-\lambda}}\right) \E\left[G_2\left(\rho^{1-\lambda}_{x,0}\right)\right]$ and $C_\star=\frac{C_+}{\sigma}$. As a consequence, \eqref{BRWmarkovprop} becomes
\begin{align*}
  &\chi\left(G_1\star G_2,K \right)\\
  &\qquad =\frac{C_\star}{ (1-\lambda)n} \int_0^K V^+(z)dz \E\left[G_1\left(\ns_t; t\in[0,\lambda]\right)\psi\left(\tfrac{S_{m}}{\sigma\sqrt{n}}\right); \underline{S}_{m}\geq 0\right]+ o(n^{-1})\P\left(\underline{S}_{m}\geq 0\right)\\
  &\qquad =\frac{C_+C_-}{\sigma(1-\lambda)\sqrt{\lambda} n^{3/2}}\int_0^K V^+(z)dz \E\left[\left.G_1\left(\tfrac{S_{\floor{nt}}}{\sigma\sqrt{n}}; t\in[0,\lambda]\right)\psi\left(\tfrac{S_{m}}{\sigma\sqrt{n}}\right)\right| \underline{S}_{m}\geq 0\right]+o(n^{-3/2}),
\end{align*}
where the last equality is a consequence of \eqref{eq:probstaypositive}.

Using \eqref{eq:cvtomeander}, conditionally on $\{\underline{S}_m=\mS_{\floor{\lambda n}} \geq 0\}$, the normalized random walk $\ns^{( n)}$ converges in law to a Brownian meander of length $\lambda$, written $\calM=(\calM(t), 0 \leq r \leq \lambda)$. Therefore, uniformly in $y\in[0,r_n]$,
\begin{align*}
  \chi\left(G_1\star G_2, K\right)
  &=\frac{C_+C_-}{\sigma(1-\lambda)\sqrt{\lambda}n^{3/2}} \int_0^KV^+(z)dz \E\left[G_1\left(\calM\right)\psi\left(\mathcal{M}(\lambda)\right)\right] +o(n^{-3/2})\\
  &=\frac{C_+C_-}{\sigma n^{3/2}}\int_0^K V^+(z)dz\Gamma(G_1, G_2,\lambda) + o(n^{-3/2}),
\end{align*}
where we write
\[
  \Gamma(G_1, G_2,\lambda)
  = \frac{1}{\lambda^{1/2}(1-\lambda)^{3/2}} \E\left[ G_1\left(\mathcal{M}\right) 
    \mathcal{M}(\lambda) e^{-\frac{\mathcal{M}(\lambda)^2}{2(1-\lambda)}} 
    \E\left[G_2\left(\rho^{1-\lambda}_{\mathcal{M}(\lambda),0}\right)\right]\right],
\]
with $(\mathcal{M}(t);0\leq r\leq \lambda)$ and $({\rho^{1-\lambda}_{x,0}}(t); t\in[0,1], x \in \R_+)$ two independent processes. Applying Lemma~\ref{lem:gamma}, we have, uniformly in $y\in[0,r_n]$,
\begin{equation}
\label{eq:convofrw2}
  \chi(G_1\star G_2;K) = \frac{C_1}{n^{3/2}} \int_0^K V^+(z)dz\times \E\left(G_1 \star G_2(\epsilon)\right) + o(n^{-3/2}),
\end{equation}
where $C_1:=\frac{C_+C_-}{\sigma}\sqrt{\frac{\pi}{2}}$, which leads to (\ref{eq:convofrwg+g}), therefore concluding the proof.
\end{proof}

This lemma can be extended, using standard computations, to the following estimate, which enables to choose the starting point uniformly in $[0,r_n]$.
\begin{lemma}
\label{cor:rw}
Under the hypotheses of Lemma~\ref{lem:rw}, we have
\begin{multline}
  \label{eq:convofrw+a}
  \lim_{n \to \infty} \sup_{a, y \in [0,r_n]} \bigg| n^{3/2} \E\left[ F\left(\ns\right) f(S_n-y) ; \mS_n \geq -a , \mS_{[\lambda n, n]}  \geq y \right] \\
  - C_1 V^-(a) \E\left[ F(\epsilon) \right] \int_{\r_+} f(x) V^+(x) dx\bigg| = 0.
\end{multline}
\end{lemma}

\begin{proof} Again, by Lemma A.1 of \cite{Pain}, we can suppose that $F\in\calC_b^u(\calD)$.
Up to decomposing $f$ into its positive and its negative part, we assume without loss of generality that $f \geq 0$.
We set
\begin{equation}
  \chi_a(F,f) := \E\left[ F(\ns) f(S_n-y) ; \mS_n \geq -a, \mS_{[\lambda n,n]} \geq y \right].
\end{equation}
Decomposing with respect to the first time at which the random walk hits its minimum, we prove that uniformly in $a\in[0,r_n]$, $\chi_a(F,f)\approx V^-(a)\chi(F,f)$. Let $\tau:=\inf\{0\leq k\leq n : S_k=\mS_n \}$, we show that $\tau \leq \sqrt{n}$ with high probability. By Markov property at time $k$, we have
\begin{align*}
\chi_a(F,f)
 &=\sum_{k=0}^n\E\left(F(\ns)f(S_n-y); \tau=k, \mS_n\geq-a, \mS_{[\lambda n, n]}\geq y\right)\\
  &\leq \sum_{k=0}^n\E\left(f(S_n-y); \tau=k, \mS_n\geq-a, \mS_{[\lambda n, n]}\geq y\right)\\
  &\leq \sum_{k=0}^n\E\left(\zeta(S_k, n-k) \ind{\mS_{k-1}>S_k\geq-a} \right),
\end{align*}
where $\zeta(x,n-k):=\E(f(S_{n-k}-y+x);\mS_{n-k}\geq0, \mS_{[\lambda n-k,n-k]}\geq y-x)$. 

On the one hand, observe that
\begin{align*}
  \zeta(x,n-k)\leq & \E\left(f(S_{n-k}-y+x)\ind{S_{n-k}\geq y-x};\mS_{n-k}\geq0\right)\\
  \leq & \sum_{j=0}^{\infty} \E\left(f(S_{n-k}-y+x)\ind{S_{n-k}\in [y-x+j, y-x+j+1]};\mS_{n-k} \geq 0 \right)\\
  \leq & \sum_{j=0}^{\infty} \hat{f}(j) \P\left( S_{n-k} \in [y-x+j, y-x+j+1], \mS_{n-k} \geq 0 \right),
\end{align*}
which, by \eqref{eq:elementary}, is bounded by
\[
 c_1\sum_{j=0}^\infty \hat{f}(j)\frac{(j+y-x+2)}{(n-k)^{3/2}}\leq \frac{c_1}{(1-\lambda)^{3/2}n^{3/2}}2(1+y-x)\sum_{j=0}^\infty (1+j)\hat{f}(j).
\]
As $\int x\hat{f}(x)dx<\infty$, uniformly in $a,y \in [0,r_n]$, $x\in[-a,0]$ and $k \leq \lambda n$, we have
\begin{equation}
  \label{eq:uppofzeta}
  \zeta(x,n-k)\leq c_2(1+y+a) n^{-3/2}.
\end{equation}

On the other hand, by \eqref{eq:elementary}, $\P(\mS_{k-1}>S_k\geq-a)\leq c_1(1+a)^2 k^{-3/2}.$ As a consequence, writing $k_n = \floor{\sqrt{n}}$, we have
\begin{align*}
  \sum_{k=k_n+1}^{\lambda n} \E\left(F(\ns)f(S_n-y); \tau=k, \mS_n\geq-a, \mS_{[\lambda n, n]}\geq y\right)
  \leq &c_3 n^{-3/2}(1+y+a)(1+a)^2\sum_{k=k_n+1}^{\lambda n}k^{-3/2}\\
  \leq & c_4\frac{(1+\log n)^3}{k_n^{1/2}}n^{-3/2}.
\end{align*}
Thus
\begin{align*}
  \chi_a(F, f) 
  &= \sum_{k=0}^{\lambda n} \E\left(F(\ns)f(S_n-y); \tau=k, \mS_n\geq-a, \mS_{[\lambda n, n]}\geq y\right)\\
  &= \sum_{k=0}^{k_n} \E\left(F(\ns)f(S_n-y); \tau=k, \mS_n\geq-a, \mS_{[\lambda n, n]}\geq y\right)+o(n^{-3/2}).
\end{align*}

We now prove that $\max_{k \leq \tau} S_k \leq n^{1/4}$ with high probability. Let $M>0$, by Markov property and \eqref{eq:uppofzeta},
\begin{align*}
  &\sum_{k=0}^{k_n} \E\left(F(\ns)f(S_n-y); \tau=k, \max_{j\leq k}S_j\geq M, \mS_n\geq-a, \mS_{[\lambda n, n]}\geq y\right)\\
  \leq & \sum_{k=0}^{k_n} \E\left(\zeta(S_k, n-k)\ind{\mS_{k-1}>S_k\geq-a; \max_{j\leq k}S_j\geq M} \right)\\
  \leq & c_2 n^{-3/2}(1+y+a) \sum_{k=0}^{k_n} \P\left(\mS_{k-1}>S_k\geq-a; \max_{j\leq k}S_j\geq M\right).
\end{align*}
We recall that $(\tau^-_k, H^-_k)_{k\geq0}$ are the strict descending epochs and heights of $(S_n)$. For all $k \geq 1$, the sequence $\left(\left\{S_{n-\tau^-_{k-1}} + H^-_{k-1}, n \leq \tau^-_k-\tau^-_{k-1}\right\}, k \geq 0\right)$ is i.i.d. Letting
\[M_k^-:=\max\{S_n+H^-_{k-1}, \tau^-_{k-1} \leq n \leq \tau^-_k\},\]
we deduce that $(H^-_k-H^-_{k-1},M_k^-)$ is i.i.d. Consequently,
\begin{align*}
  \sum_{k=0}^{k_n} &\P\left(\mS_{k-1}>S_k\geq-a; \max_{j\leq k}S_j\geq M\right)\\
  &\leq \P(M_1^-\geq M)+\sum_{k\geq1}\P(H^-_k\leq a, M_1^-<M,\ldots, M_k^-<M+H^-_{k-1}, M_{k+1}^->M+H^-_k)\\
  &\leq \P(M_1^-\geq M)+\sum_{k\geq1}\P(H^-_k\leq a,M_{k+1}^->M+H^-_k)\\
  &\leq \P(M_1^-\geq M)+\sum_{k\geq1}\P(H^-_k\leq a)\P(M_1^->M)\\
  &\leq V^-(a)\P(M_1^->M).
\end{align*}
According to Corollary 3 in~\cite{Don85}, $\P(M_1^->n)= \frac{c}{n}+o(n^{-1})$. Taking $M=n^{1/4}$ yields
\[\sum_{k=0}^{k_n} \E\left(F(\ns)f(S_n-y); \tau=k, \max_{j\leq k}S_j\geq n^{1/4}, \mS_n\geq-a, \mS_{[\lambda n, n]}\geq y\right)=o(n^{-3/2}).\]
Finally, by uniform continuity of $F$, we have
\begin{align}
  \label{eq:chia}
  \chi_a(F, f)=&\sum_{k=0}^{k_n} \E\Big(F(\ns)f(S_n-y); \tau=k, \max_{j\leq k}S_j\leq n^{1/4}, \mS_n\geq-a, \mS_{[\lambda n, n]}\geq y\Big)+o(n^{-3/2})\\
  =&\sum_{k=0}^{k_n} \E\Big(\zeta(S_k,n-k,F)\ind{\mS_{k-1}>S_k\geq-a} \Big)+o(n^{-3/2}),\nonumber
\end{align}
where 
\[
\zeta(x,n-k,F):=\E\left(F\Big(\frac{S_{\floor{(n-k)t}}}{\sqrt{\sigma^2(n-k)}}, t\in[0,1]\Big)f(S_{n-k}-y+x); \mS_{n-k}\geq0, \mS_{[\lambda n-k,n-k]}\geq y-x\right).
\]
Observe that for $k\leq \sqrt{n}$, the asymptotic behavior of $\zeta(x,n-k,F)$ follows from that of $\chi(F,f)$. It follows from \eqref{eq:convofrw} that uniformly in $k\leq k_n$, $x\in[-a,0]$ and $a, y\in[0,r_n]$,
\begin{equation*}
  \zeta(x,n-k,F)=\frac{C_1}{n^{3/2}} \E\left(F(\epsilon)\right)\int_0^\infty f(z) V^+(z)dz +o(n^{-3/2}).
\end{equation*}
Going back to \eqref{eq:chia}, we have
\begin{align*}
  \chi_a(F, f)=&\sum_{k=0}^{k_n} \E\left( \zeta(S_k,n-k,F) \ind{\mS_{k-1}>S_k\geq-a}\right)+o(n^{-3/2})\\
  =&\frac{C_1}{n^{3/2}}\E\left(F(\epsilon)\right) \int_0^\infty f(z) V^+(z)dz \sum_{k=0}^{k_n} \P\left(\mS_{k-1}>S_k\geq-a\right)+o(n^{-3/2}).
\end{align*}
Observe also that $\sum_{k=0}^{\infty} \P\left(\mS_{k-1}>S_k\geq-a\right)=V^-(a)$ and that uniformly in $a\in[0,r_n]$,
\[\sum_{k=k_n+1}^{\infty} \P\left(\mS_{k-1}>S_k\geq-a\right)=o_n(1).\]
We conclude that uniformly in $y,a\in[0,r_n]$,
\begin{equation}
  \lim_{n \to \infty} n^{3/2}\chi(F, f) = C_1V^-(a)\E\left(F(\epsilon)\right) \int_{0}^\infty f(z)V^+(z)dz ,
\end{equation}
which ends the proof.
\end{proof}

\section{Laplace transform of the Gibbs measure}
\label{sec:taildecay}

We recall that for a branching random walk $(V(u), u \in \T)$ and $\beta > 1$,
\[ \tilde{\mu}_{n,\beta}(F) = n^{\frac{3\beta}{2}}\sum_{|u|=n} e^{-\beta V(u)} F(H_n(u)). \]
This section is devoted to the computation of the Laplace transform of $\tilde{\mu}_{n,\beta}(F)$, which is closely related to the already known estimates on the minimal displacement of the branching random walk. Therefore, we define $M_n$ as the smallest occupied position in the $n$-th generation, i.e.,
\begin{equation}
  M_n:=\inf_{|u|=n} V(u),
\end{equation}
with the convention $\inf\emptyset:=\infty$. We denote by $m^{(n)}$ an individual chosen uniformly at random in the set $\{u: |u|=n, V(u)=M_n\}$ of leftmost individuals at time $n$.

The rest of this section is devoted to the proof of the following result.
\begin{proposition}
\label{tailprop}
Let $\beta > 1$, under \eqref{eqn:boundary}, \eqref{eqn:variance} and \eqref{eqn:integrability}, there exists $C_\beta > 0$ such that for all non-negative $F\in\calC_b(\calD)$ and $\varepsilon > 0$, there exists $(A,N) \in \R^+ \times \N$ such that
\begin{equation}
  \label{sansDelta2}
  \sup_{n \geq N} \sup_{x \in [A,\frac{3}{2}\log n - A]}\left| \frac{e^x}{x}\E\left[ 1 - \exp\left(-e^{-\beta x} \tilde{\mu}_{n,\beta}(F)\right) \right] - C_\beta \E\left(F(\epsilon)^{\frac{1}{\beta}}\right) \right|\leq \varepsilon,
\end{equation}
where $\epsilon$ is a standard Brownian excursion.
\end{proposition}

Observe that if $F=\theta \in \R^+$ is a constant, Proposition~\ref{tailprop} is: For all $\varepsilon > 0$, there exists $(A,N) \in \R^+ \times \N$ such that
\begin{equation}
  \label{sansF}
  \sup_{n \geq N} \sup_{x \in [A,\frac{3}{2}\log n - A]} \left|\frac{e^x}{x}\E\left(1-\exp\left(-  \theta e^{-\beta x} \tilde{\mu}_{n,\beta}(1)\right) \right)-C_\beta \theta^{\frac{1}{\beta}} \right|\leq \varepsilon,
\end{equation}
which is a straightforward consequence of~\cite[Proposition 2.2]{Mad11} (applying this result with $d=1$). Therefore, it is enough to prove, using Lemma~\ref{lem:rw} that
\[\E\left[\exp\left(- e^{-\beta x} \widetilde{\mu}_{n,\beta}(F)\right)\right] \approx \E\left[\exp\left(- e^{-\beta x} \tilde{\mu}_{n,\beta}(1)F(\epsilon)\right)\right] \]
where $\epsilon$ is a Brownian excursion independent of the branching random walk. This is done in Lemma~\ref{lem:indep}.

However, to realize this substitution, we first need to restrict the space on which we compute the Laplace transform of $\tilde{\mu}_{n,\beta}$ to an event of in which there is a unique family of particles in the neighborhood of the minimal displacement at time $n$, which followed a random walk excursion. The tail of the Laplace transform of $\tilde{\mu}_{n,\beta}$ on the whole space is well-approached by its tail on that subspace. The computation of the tail of a random variable by considering it on a subspace is a fruitful technique in branching processes, and can be tracked back at least to \cite{Bramson}. This method was used by Aïdékon \cite{Aid11} to obtain a precise estimate on the tail of the maximal displacement of the branching random walk.

For all $n \in \N$, following~\cite{Aid11}, we write $a_n = \frac{3}{2} \log n$ and $a_n(z) = a_n-z$. For all $x \in \R$, $F \in \calC_b(\calD,\R^+)$ and $E$ a measurable event, we write
\[\Sigma(n,x,F):=\E\left[\exp\left(-e^{-\beta x}\widetilde{\mu}_{n,\beta}(F)\right)\right] \quad \mathrm{and} \quad
\Sigma_E(n,x,F):=\E\left[\exp\left(-e^{-\beta x} \widetilde{\mu}_{n,\beta}(F)\indset{E}\right)\right].
\]

For $\lambda \in (0,1)$, $L,L_0 \geq 0$ and $z > K_0>0$, we define the set of individuals
\begin{equation}
  J^L_{\lambda,z,K_0,L_0}(n) = \left\{u \in \T:
  \begin{array}{c} |u|=n, V(u)\leq a_n(z-L), \min_{k\leq n}V(u_k)\geq -z+K_0,\\
  \min_{\lambda n\leq k\leq n}V(u_k)\geq a_n(z+L_0)\end{array}\right\}.
\end{equation}
For simplicity, we often write $J_{\lambda,z,K_0,L_0}(n)$ instead of $J^0_{\lambda,z,K_0,L_0}(n)$. We now consider the following event
\begin{equation}
  E_n:=\{m^{(n)}\in J_{\lambda,x-\Delta,K_0,L_0}(n)\}.
\end{equation}
At the end of the section, we will choose $\Delta < L_0 \ll K_0 \ll x$, and $L \in \{0,L_0\}$. We prove in a first time that $\Sigma$ and $\Sigma_{E_n}$ are close to each other.

\begin{lemma}
\label{lem:ajouteE}
There exists $\alpha_1>0$ small enough such that for all $\varepsilon>0$, there exists $\Delta_{\varepsilon,1}\geq1$ such that such that for all $\Delta \geq \Delta_{\varepsilon,1}$, $L_0\geq 2\Delta/\alpha_1$, $x\geq 2e^{K_0+\Delta}/\varepsilon$ and $n\geq1$, we have
\begin{equation}
  0\leq \Sigma_{E_n}(n,x,F)-\Sigma(n,x,F)\leq \varepsilon xe^{-x}.
\end{equation}
\end{lemma}

\begin{proof}
Observe that
\begin{align*}
  \Sigma_{E_n}(n,x,F)=&\E\left[\exp\left(-e^{-\beta x}\widetilde{\mu}_{n,\beta}(F)\right); E_n\right)+\P\left(E_n^c\right),\\
  \Sigma(n,x,F)=&\E\left[\exp\left(-e^{-\beta x}\widetilde{\mu}_{n,\beta}(F)\right); E_n\right]+\E\left[\exp\left(-e^{-\beta x}\widetilde{\mu}_{n,\beta}(F)\right); E_n^c\right].
\end{align*}
As a consequence,
\begin{equation}
0\leq \Sigma_{E_n}(n,x,F)-\Sigma(n,x,F)=\E\left(1-\exp\{-e^{-\beta x}\widetilde{\mu}_{n,\beta}(F)\}; E_n^c\right).
\end{equation}

We observe that $1-e^{- W}=\int_0^\infty e^{-u}\ind{W\geq u}du$, thus
\begin{align*}
  \Sigma_{E_n}(n,x,F)-\Sigma(n,x,F)
  &= \E\left[\int_0^\infty e^{-u}\ind{e^{-\beta x} \widetilde{\mu}_{n,\beta}(F)\geq u}du; E_n^c\right]\\
  &= \int_0^\infty e^{-u}\P\left(e^{-\beta x} \widetilde{\mu}_{n,\beta}(F)\geq u;   E_n^c\right)du.
\end{align*}
Using the fact that $F$ is non-negative bounded, we have
\[
  \Sigma_{E_n}(n,x,F)- \Sigma(n,x,F)
  \leq \int_0^{\infty} e^{-u} \P\left( \tilde{\mu}_{n,\beta}(1) \geq \tfrac{u}{\norme{F}_\infty}e^{\beta x} ; E_n^c \right).
\]

Let $\Delta \in (1,x-1)$, as $E_n^c\subset \{M_n\geq a_n(x-\Delta)\}\cup\left(E_n^c\cap\{M_n\leq a_n(x-\Delta)\}\right)$, we have
\begin{multline}
\label{eq:ajouteEetM}
  \Sigma_{E_n}(n,x,F)- \Sigma(n,x,F)\\
  \leq \underbrace{\int_0^\infty e^{-u}\P\left(\tilde{\mu}_{n,\beta}(1)\geq \tfrac{u}{\norme{F}_\infty} e^{\beta x}; M_n\geq a_n(x-\Delta)\right)du}_{\P_\dagger}\\
  + \int_{0}^\infty e^{-u}\underbrace{\P\left(M_n\leq a_n(x-\Delta); E_n^c\right)}_{\P_\ddagger} du.
\end{multline}

On the one hand,
\begin{align}
  \P_\ddagger =& \P\left(m^{(n)}\notin J_{\lambda,x-\Delta,K_0,L_0}(n); M_n\leq a_n(x-\Delta)\right)\nonumber\\
  \leq& \P\left(\exists z: |z|=n, V(z)\leq a_n(x-\Delta), z\notin J_{\lambda,x-\Delta,K_0,L_0}(n)\right)\nonumber\\
  \leq &\left(e^{K_0}+e^{-c_6 L_0}x\right)e^{-x+\Delta}
\label{eq:ddagger}
\end{align}
applying Lemma 3.3 in A\"{i}d\'ekon~\cite{Aid11}.

On the other hand, by change of variables, 
\begin{equation}
  \P_\dagger=\int_{\R} \beta e^{-e^{\beta y}+\beta y}\underbrace{\P\left(\tilde{\mu}_{n,\beta}(1)\geq \tfrac{1}{\norme{F}_\infty} e^{\beta(x+y)}; M_n\geq a_n(x-\Delta)\right)}_{\P_\dagger(x,y)}dy.
\end{equation}
To bound $\P_\dagger(x,y)$, we use Proposition 4.6 of~\cite{Mad11} (more precisely Equation (4.15) of that article). For all $0\leq K\leq  \Delta$, one sees immediately that, for $|y|\leq K$,
\begin{align}
  \P_\dagger(x,y) &= \P\left(\tilde{\mu}_{n,\beta}(1) \geq\frac{1}{\norme{F}_\infty} e^{\beta (x+y)}; M_n\geq a_n(x-\Delta)\right)\nonumber\\
  &\leq \sum_{j\geq \Delta+y}\P\left(\tilde{\mu}_{n,\beta}(1)\geq \frac{1}{\norme{F}_\infty} e^{\beta (x+y)}; M_n-a_n(0)\in[j-(x+y); (j+1)-(x+y)]\right)\nonumber\\
  &\leq c_7 (x+y)e^{-(x+y)} e^{-\alpha (\Delta+y)}\nonumber\\
  &\leq c_8 xe^{-x} e^{(1+\alpha)K-\alpha \Delta}  \label{eq:daggerinK},
\end{align}
where $\alpha >0$ is a given constant depending only on the law of the branching random walk.
In the same way, for $|y|>K$, we have
\begin{align}
  \P_\dagger(x,y) &\leq \P\left( \tilde{\mu}_{n,\beta}(1) \geq \frac{1}{\norme{F}_\infty} e^{\beta (x+y)}\right)\nonumber\\
  &\leq c_9(x+y)e^{-(x+y)} \ind{x+y \geq 1} + \ind{y+x \leq 1}\nonumber\\
  &\leq c_9 xe^{-x} \ind{y>K} + c_9 e^{-(x+y)}\ind{-K>y\geq 1-x} + \ind{x+y \leq 1}.
  \label{eq:daggeroutK}
\end{align}
Combining \eqref{eq:daggerinK} with \eqref{eq:daggeroutK} yields
\begin{align*}
  \P_\dagger &\leq c_8 xe^{-x} e^{(1+\alpha)K-\alpha \Delta} + c_9 xe^{-x}\int_{K}^{\infty} \beta e^{-e^{\beta y}+\beta y}dy\nonumber\\
  &\qquad \qquad +c_9 xe^{-x}\int_{\{1-x \leq y < -K\}}\beta e^{-e^{\beta y}+\beta y-y}dy+\int_{-\infty}^{1-x} \beta e^{-e^{\beta y}+\beta y}dy \nonumber\\
&\leq c_8 xe^{-x} e^{(1+\alpha)K-\alpha \Delta}+c_{10} xe^{-x}\left(e^{(1-\beta)K}+e^{(1-\beta)x}\right).
\end{align*}
Take $K=\frac{\alpha\Delta}{\alpha+\beta}$. There exists $c_{11}>0$ such that
\begin{equation}
  \label{eq:dagger}
  \P_\dagger\leq (c_8 + c_{10}) xe^{-x} e^{-\frac{(\beta-1)\alpha\Delta}{\alpha+\beta}} +c_{10} xe^{-x}e^{(1-\beta)x} \leq c_{11} x e^{-x} e^{-\frac{(\beta-1)\alpha\Delta}{\alpha+\beta}},
\end{equation}
for all $x \geq 2 e^{K_0+\Delta}/\varepsilon \geq 1$.

Using \eqref{eq:ddagger} and \eqref{eq:dagger}, inequality \eqref{eq:ajouteEetM} becomes
\begin{equation}
  \Sigma_{E_n}(n,x,F)-\Sigma(n,x,F) \leq c_{11} x e^{-x} e^{-\frac{(\beta-1)\alpha\Delta}{\alpha+\beta}}+\left(e^{K_0}+e^{-c_6 L_0}x\right)e^{-x+\Delta}.
\end{equation}
We set $\alpha_1:=\min\{\frac{(\beta-1)\alpha}{\alpha+\beta},c_6\}$ and $L_0\geq 2\Delta/\alpha_1$, we have
\begin{equation*}
  \Sigma_{E_n}(n,x,F)-\Sigma(n,x,F) \leq c_{12} x e^{-x} e^{-\alpha_1\Delta}+\frac{e^{K_0+\Delta}}{x} xe^{-x}.
\end{equation*}
Since $\alpha_1>0$, for all $\varepsilon>0$, there exists $\Delta_{\varepsilon,1}>1$ such that $c_{12}e^{-\alpha_1\Delta_{\varepsilon,1}}\leq \varepsilon/2$. For all $\Delta\geq\Delta_{\varepsilon,1}$ and $x\geq 2e^{K_0+\Delta}/\varepsilon$ we obtain finally that
\begin{equation}
  \Sigma_{E_n}(n,x,F)-\Sigma(n,x,F) \leq\varepsilon xe^{-x},
\end{equation}
which ends the proof.
\end{proof}

In what follows, we prove that on the set $E_n$, the individuals who make the most important contribution to $\tilde{\mu}_{n,\beta}(F)$ are the ones who are geographically close to $m^{(n)}$. For any $L\geq1$, let
\begin{equation}
  \tilde{\mu}^L_{n,\beta}(F) := n^{3\beta/2} \sum_{u \in J_{\lambda,x-\Delta, K_0,L_0}^L(n)} e^{-\beta V(u)} F(H^{(n)}(u)) \quad \mathrm{and} \quad \widetilde{W}^{L}_{n,\beta} := \tilde{\mu}^L_{n,\beta}(1).
\end{equation}
In the same way as above, for any measurable event $E$, we denote
\begin{equation}
\Sigma^L_E(n,x,F):=\E\left[\exp\left(-e^{-\beta x}\widetilde{\mu}^L_{n,\beta}(F)\indset{E}\right)\right].
\end{equation}

We now prove the following lemma.
\begin{lemma}
\label{lem:ajouteL}
There exists $\alpha_2>0$ such that for all $\varepsilon > 0$ there exists $\Delta_{\varepsilon,2} \geq 1$ such that for all $\Delta \geq \Delta_{\varepsilon,2}$, $L=L_0\geq 2\Delta/\alpha_2$, $x\geq 2e^{K_0+\Delta}/\varepsilon$ and $n\geq1$, we have
\begin{equation}
  0\leq \Sigma^L_{E_n}(n,x,F)-\Sigma_{E_n}(n,x,F)\leq \varepsilon x e^{-x}.
\end{equation}
\end{lemma}

\begin{proof}
As $\tilde{\mu}_{n, \beta}(F) \geq \tilde{\mu}^L_{n, \beta}(F)$, we have
\begin{equation*}
  \Sigma^L_{E_n}(n,x,F)-\Sigma_{E_n}(n,x,F)=\E\left[\exp\left(-e^{-\beta x}\widetilde{\mu}^L_{n,\beta}(F)\right)-\exp\left(-e^{-\beta x}\widetilde{\mu}_{n,\beta}(F)\right); E_n\right]\geq 0,
\end{equation*}
We observe that, for all $0 \leq W_1 \leq W_2$,
\begin{align*}
\left(e^{-W_1}- e^{-W_2}\right)\ind{W_2-W_1\geq 0} \leq &|W_1-W_2|\ind{0\leq W_2-W_1\leq \delta}+\ind{W_2-W_1> \delta}\\
\leq &\delta + \ind{W_2-W_1> \delta}.
\end{align*}
Applying this inequality with $\delta=e^{-\beta\Delta}$, $W_1=\tilde{\mu}^L_{n, \beta}(F)$ and $W_2=\tilde{\mu}_{n, \beta}(F)$ gives
\[
  \Sigma^L_{E_n}(n,x,F)-\Sigma_{E_n}(n,x,F)
  \leq e^{-\beta\Delta}\P\left(E_n\right)+\P\left(\widetilde{\mu}_{n,\beta}(F)-\widetilde{\mu}^L_{n,\beta}(F)\geq e^{\beta(x-\Delta)}; E_n\right).
\]

As $E_n\subset\{M_n\leq a_n(x-\Delta)\}$, we have
\begin{equation}\label{eq:ajouteL}
  \Sigma^L_{E_n}(n,x,F)-\Sigma_{E_n}(n,x,F)\leq e^{-\beta\Delta}\P\Big(M_n\leq a_n(x-\Delta)\Big)+\P_\diamond,
\end{equation}
where
\begin{equation*}
  \P_\diamond:= \P\left(n^{3\beta/2}\sum_{|u|=n}\ind{u\notin J_{\lambda,x-\Delta,K_0,L_0}^{L}(n)} e^{-\beta V(u)} \geq e^{\beta(x-\Delta)}; M_n\leq a_n(x-\Delta)\right).
\end{equation*}
From \eqref{eq:ajouteL}, on the one hand we recall (see e.g. the proof of the upper bound of Theorem~4.1 in~\cite{Mall}) there exists $c_{12}>0$ such that for all $x \geq \Delta +1$,
\begin{equation*}
  \P\left(M_n\leq a_n(x-\Delta)\right) \leq c_{12} (x-\Delta)e^{-(x-\Delta)}.
\end{equation*}
On the other hand, by Proposition 4.6 of~\cite{Mad11}, there exists $\alpha_2\in(0,\beta-1)$ such that for $L=L_0$,
\begin{equation}
  \label{eq:diamond}
  \P_\diamond\leq e^{K_0+\Delta}e^{-x}+ c_{13} xe^{-x}e^{-\alpha_2 L_0+\Delta}.
\end{equation}
As a consequence,
\begin{equation*}
  \Sigma^L_{E_n}(n,x,F)-\Sigma_{E_n}(n,x,F)\leq c_{14} xe^{-x}\left( e^{-\alpha_2\Delta}+e^{-\alpha_2 L_0+\Delta}\right)+e^{K_0+\Delta}e^{-x}.
\end{equation*}
For any $\varepsilon>0$, there exists $\Delta_{\varepsilon,2}>0$ such that $c_{14}e^{-\alpha_2\Delta_{\varepsilon,2}}\leq \varepsilon/4$. We set $\Delta\geq\Delta_{\varepsilon,2}$, $L_0\geq2\Delta/\alpha_2$ and $x\geq 2e^{K_0+\Delta}/\varepsilon$, and obtain that
\begin{equation}
0\leq\Sigma^{L_0}_{E_n}(n,x,F)-\Sigma_{E_n}(n,x,F)\leq \varepsilon xe^{-x}
\end{equation}
which ends the proof.
\end{proof}

Recall that $m^{(n)}$ is uniformly chosen from the set of leftmost individuals at time $n$. For any $1\leq k\leq n$, we use $m_k^{(n)}$ to represent the ancestor of $m^{(n)}$ at generation $k$. We prove now that the individuals who make significant contributions to $\tilde{\mu}$ are the close relatives of $m^{(n)}$. We write, for $k \leq n$
\begin{equation}
  \widehat{\mu}^{L}_{n,k,\beta}(F):=n^{3\beta/2}\sum_{u\in J_{\lambda,x-\Delta,K_0,L_0}^{L}(n)}\ind{u\geq m_{k}^{(n)}} e^{-\beta V(u)}F(H^{(n)}(u)),
\end{equation}
and for $E$ a measurable event
\begin{equation}
  \widehat{\Sigma}^{L}_{E}(n,k,x,F):=\E\left[\exp\left(-e^{-\beta x}\widehat{\mu}^{L}_{n,k,\beta}(F)\indset{E}\right)\right].
\end{equation}
 
\begin{lemma}
\label{lem:compteB}
For all $\varepsilon>0$ and $L_0\geq1$, there exist $K=K_{\varepsilon,L_0}>0$, $B=B_{\varepsilon,L_0}\geq1$ and $N=N_{\varepsilon,L_0}\geq1$ such that for all $K_0\geq K+L_0$, $n\geq N$ and $b\geq B$,
\begin{equation}
  0 \leq \widehat{\Sigma}^{L_0}_{E_n}(n,n-b,x,F)-\Sigma^{L_0}_{E_n}(n,x,F)\leq \varepsilon x e^{-x}.
\end{equation}
\end{lemma}

Before giving the proof of Lemma~\ref{lem:compteB}, we state a result about the branching random walk under $\hat{\P}$. Recall that $(\omega_k; k\geq0)$ is the spine of $\mathbb{T}$. For any integer $b\geq0$, we define
\begin{equation}
\xi_n(z,L,b):=\{\forall k\leq n-b, \min_{u\geq \omega_k; |u|=n}V(u)\geq a_n(z)+L\}.
\end{equation}

\begin{fact}
\label{fact:splitting}
For any $\eta>0$ and $L>0$, there exists $K(\eta)>0$, $B(L,\eta)\geq1$ and $N(\eta)\geq1$ such that for any $b\geq B(L,\eta)\geq1$, $n\geq N(\eta)$ and $z\geq K\geq K(\eta)+L$,
\begin{equation}
\hat{\P}\left(\xi_n(z,L,b)^c, \omega_n\in J_{\lambda,z,K,L}(n)\right)\leq \eta (1+L)^2(1+z-K) n^{-3/2}.
\end{equation}
\end{fact}
Fact~\ref{fact:splitting} is a slight refinement of Lemma 3.8 in~\cite{Aid11}, so we feel free to omit its proof. Using this result, we prove Lemma~\ref{lem:compteB} as follows.

\begin{proof}
As $\widehat{\mu}^{L_0}_{n,k,\beta}(F)\leq\widetilde{\mu}^{L_0}_{n,\beta}(F)$, we have $\widehat{\Sigma}^{L_0}_{E_n}(n,n-b,x,F)-\Sigma^{L_0}_{E_n}(n,x,F)\geq0$. We also observe that
\begin{equation*}
\widehat{\Sigma}^{L_0}_{E_n}(n,n-b,x,F)-\Sigma^{L_0}_{E_n}(n,x,F)
  =\E\left[\exp\left(-e^{-\beta x}\widehat{\mu}^{L_0}_{n,n-b,\beta}(F)\right)-\exp\left(-e^{-\beta x}\widetilde{\mu}^{L_0}_{n,\beta}(F)\right); E_n\right].
\end{equation*}
By change of measures, we have
\begin{align*}
  &\widehat{\Sigma}^{L_0}_{E_n}(n,n-b,x,F)-\Sigma^{L_0}_{E_n}(n,x,F)\\
  = & \hat{\E}\left[\frac{\exp\left(-e^{-\beta x}\widehat{\mu}^{L_0}_{n,n-b,\beta}(F)\right)-\exp\left(-e^{-\beta x}\widetilde{\mu}^{L_0}_{n,\beta}(F)\right)}{W_n}; m^{(n)}\in J_{\lambda,x-\Delta,K_0,L_0}(n)\right]\\
  = & \hat{\E}\left[\frac{e^{V(\omega_n)}\ind{V(\omega_n)=M_n, \omega_n\in J_{\lambda,x-\Delta,K_0,L_0}(n)}}{\sum_{|u|=n}\ind{V(u)=M_n}}\left[\exp\left(-e^{-\beta x}\mathring{\mu}^{L_0}_{n,n-b,\beta}(F)\right)-\exp\left(-e^{-\beta x}\widetilde{\mu}^{L_0}_{n,\beta}(F)\right)\right]\right],
\end{align*}
where 
\[\mathring{\mu}^{L_0}_{n,k,\beta}(F):=n^{3\beta/2}\sum_{u\in J_{\lambda,x-\Delta,K_0,L_0}^{L_0}(n)}\ind{u\geq \omega_k}e^{-\beta V(u)}F(H_n(u)).
\]

Observe that $0\leq \exp\{-e^{-\beta x}\mathring{\mu}^{L_0}_{n,n-b,\beta}(F)\}-\exp\{-e^{-\beta x}\widetilde{\mu}^{L_0}_{n,\beta}(F)\}\leq1$. Moreover, on the event $\xi_n(x-\Delta,L_0,b)$, we have $\mathring{\mu}^{L_0}_{n,n-b,\beta}(F)=\widetilde{\mu}^{L_0}_{n,\beta}(F)$. Therefore,
\begin{align}
  \widehat{\Sigma}^{L_0}_{E_n}(n,n-b,x,F)-\Sigma^{L_0}_{E_n}(n,x,F)
  \leq&\hat{\E}\left[\frac{e^{V(\omega_n)}\ind{V(\omega_n)=M_n, \omega_n\in J_{\lambda,x-\Delta,K_0,L_0}(n)}}{\sum_{|u|=n}\ind{V(u)=M_n}}; \xi_n^c(x-\Delta,L_0,b)\right]\nonumber\\
  \leq &\hat{\E}\left[e^{V(\omega_n)}\ind{V(\omega_n)=M_n, \omega_n\in J_{\lambda,x-\Delta,K_0,L_0}(n)}; \xi_n^c(x-\Delta,L_0,b)\right]\nonumber\\
\leq& n^{3/2} e^{-x+\Delta}\hat{\P}\left(\xi_n^c(x-\Delta,L_0,b), \omega_n\in J_{\lambda,x-\Delta,K_0,L_0}(n)\right).\label{eq:ajoutezeta}
\end{align}
Applying Fact~\ref{fact:splitting} to $\eta=\varepsilon e^{-\Delta}/(1+L_0)^2$ shows that
\begin{equation}
  \widehat{\Sigma}^{L_0}_{E_n}(n,n-b,x,F)-\Sigma^{L_0}_{E_n}(n,x,F)\leq \varepsilon x e^{-x},
\end{equation}
holds for $n\geq N(\eta)$, $b\geq B(L_0, \eta)$ and $x>K_0\geq K(\eta)+L_0$, which ends the proof.
\end{proof}

We now study $\hat{\Sigma}^{L_0}_{E_n}(n,n-b,x,F)$, to prove Proposition~\ref{tailprop}. We begin with the following estimate, which brings out the Brownian excursion.
\begin{lemma}
\label{lem:indep}For any $\varepsilon>0$, set $\Delta=\Delta_\varepsilon:=\Delta_{\varepsilon,1}\vee \Delta_{\varepsilon,2}$ and $L_0=\frac{2\Delta}{\alpha_1\wedge\alpha_2}$. Let $K=K_{\varepsilon,L_0}>0$, $B=B_{\varepsilon,L_0}\geq1$ and $N=N_{\varepsilon,L_0}\geq1$ as in Lemma~\ref{lem:compteB}. For all $K_0\geq K+L_0$, $n\geq N$ and $b\geq B$, there exists $n_\varepsilon\geq N$ such that for all $n\geq n_\varepsilon$ and $x\geq 2e^{K_0+\Delta}/\varepsilon$,
\begin{equation}
  \left\vert\widehat{\Sigma}^L_{E_n}(n,n-b,x,F)-\E\left[\exp\left(-e^{-\beta x}\tilde{\mu}_{n,\beta}(1) F(\epsilon)\right)\right]\right\vert\leq \varepsilon xe^{-x}.
\end{equation}
\end{lemma}

\begin{proof}
By change of measures, we have
\begin{align*}
  \widehat{\Sigma}^{L_0}_{E_n}(n,n-b,x,F)
  &=\E\left[\exp\left(-e^{-\beta x}\widehat{\mu}^{{L_0}}_{n,n-b,\beta}(F)\indset{E_n}\right)\right]\\
  &=\hat{\E}\left[ \frac{e^{V(\omega_n)}\ind{V(\omega_n)=M_n, \omega_n\in J_{\lambda,x-\Delta,K_0,L_0}(n)}}{\sum_{|u|=n}\ind{V(u)=M_n}}\exp\left(-e^{-\beta x}\mathring{\mu}^{L_0}_{n,n-b,\beta}(F)\right)\right] + \P(E_n^c).
\end{align*}
First, we are going to compare it with $\E\left[\exp\left(-e^{-\beta x}\widehat{\mu}^{L_0}_{n,n-b,\beta}(1) F(\epsilon)\mathbf{1}_{E_n}\right)\right]$, which equals to
\[
\hat{\E}\left[ \frac{e^{V(\omega_n)}\ind{V(\omega_n)=M_n, \omega_n\in J_{\lambda,x-\Delta,K_0,L_0}(n)}}{\sum_{|u|=n}\ind{V(u)=M_n}}\exp\left(-e^{-\beta x}\mathring{\mu}^{L_0}_{n,n-b,\beta}(1)F(\epsilon)\right)\right] + \P(E_n^c).
\]
The strategy is to show that 
\begin{align*}
\widehat{\Upsilon}^{L_0}_{E_n}(n,n-b,x,F):=&\widehat{\Sigma}^{L_0}_{E_n}(n,n-b,x,F)-\P(E_n^c)\\
\textrm{ and }\widehat{\Upsilon}^{L_0}_{E_n}(n,n-b,x,F(\epsilon)):=&\E\left[\exp\left(-e^{-\beta x}\widehat{\mu}^{L_0}_{n,n-b,\beta}(1) F(\epsilon)\mathbf{1}_{E_n}\right)\right]-\P(E_n^c)
\end{align*}
are both close to the same quantity as $n\rightarrow\infty$. Then we compare \[\E\left[\exp\left(-e^{-\beta x}\widehat{\mu}^{L_0}_{n,n-b,\beta}(1) F(\epsilon)\mathbf{1}_{E_n}\right)\right]\textrm{  with  } \E\left[\exp\left(-e^{-\beta x}\tilde{\mu}_{n,\beta}(1) F(\epsilon)\right)\right].\]
We set
\begin{equation}
  Z:=\frac{e^{V(\omega_n)}\ind{V(\omega_n)=M_n, \omega_n\in J_{\lambda,x-\Delta,K_0,L_0}(n)}}{\sum_{|u|=n}\ind{V(u)=M_n}}
  \quad \mathrm{and} \quad
  Z_b:=\frac{e^{V(\omega_n)}\ind{V(\omega_n)=M_n, \omega_n\in J_{\lambda,x-\Delta,K_0,L_0}(n)}}{\sum_{|u|=n}\ind{V(u)=M_n, u\geq \omega_{n-b}}},
\end{equation}
so that $\widehat{\Upsilon}^{L_0}_{E_n}(n,n-b,x,F)=\hat{\E}\left[Ze^{-e^{-\beta x}\mathring{\mu}^{L_0}_{n,n-b,\beta}(F)}\right]$.

Under the measure $\hat{\P}$, on the set $\xi_n(x-\Delta,{L_0},b)$, we have $Z=Z_b$, thus
\begin{multline}\label{eq:upsilon}
  \widehat{\Upsilon}^{L_0}_{E_n}(n,n-b,x,F)\\
  =\hat{\E}\left[Z_be^{-e^{-\beta x}\mathring{\mu}^{L_0}_{n,n-b,\beta}(F)}; \xi_n(x-\Delta, L_0,b)\right]+\hat{\E}\left[Z\exp\left(-e^{-\beta x}\mathring{\mu}^{L_0}_{n,n-b,\beta}(F)\right);\xi_n^c(x-\Delta,{L_0},b)\right].
\end{multline}

Recall that under $\hat{\P}$, we have
\[ \mathring{\mu}^{L_0}_{n,n-b,\beta}(F)=n^{3\beta/2}\sum_{u\in J_{\lambda,x-\Delta,K_0,L_0}^{{L_0}}(n)}\ind{u\geq \omega_{n-b}}e^{-\beta V(u)}F(H^n(u)). \]
For $n\gg b$ large and $|u|=n$,  we define the path $\widetilde{H}^n_s(u) = \frac{V(u_{\floor{ ns \wedge (n-b)}})}{\sigma\sqrt{n}},\forall s\in[0,1]$. Observe that, for all $u\geq\omega_{n-b}$, $\tilde{H}^n(u)$ is identical to $\tilde{H}^n(\omega_n)$. For
 all $\varepsilon_0>0$, let
\begin{equation}
  X_{F,\varepsilon_0}:=F\left(\widetilde{H}^n_s(\omega_n); s\in[0,1]\right)\vee\varepsilon_0.
\end{equation}

We prove that $\widehat{\Upsilon}^{L_0}_{E_n}(n,n-b,x,F)$ is close to $\hat{\E}\left[Z_b \exp\left(-e^{-\beta x}\mathring{\mu}^{L_0}_{n,n-b,\beta}(1)\times X_{F,\varepsilon_0}\right)\right]$. It follows from \eqref{eq:upsilon} that
\begin{multline}
\label{eq:ChangeF}
  \left|\widehat{\Upsilon}^{L_0}_{E_n}(n,n-b,x,F)-\hat{\E}\left[Z_b \exp\left(-e^{-\beta x}\mathring{\mu}^{L_0}_{n,n-b,\beta}(1)\times X_{F,\varepsilon_0}\right)\right]\right|\\
\leq \left|\hat{\E}\left[Z_b\left( \exp\left(-e^{-\beta x}\mathring{\mu}^{L_0}_{n,n-b,\beta}(F)\right)- \exp\left(-e^{-\beta x}\mathring{\mu}^{L_0}_{n,n-b,\beta}(1)\times X_{F,\varepsilon_0}\right) \right); \xi_n(x-\Delta, L_0, b) \right] \right|\\
\hat{\E}\left[Z\exp\left(-e^{-\beta x}\mathring{\mu}^{L_0}_{n,n-b,\beta}(1)\times X_{F,\varepsilon_0}\right)+Z_b\exp\left(-e^{-\beta x}\mathring{\mu}^{L_0}_{n,n-b,\beta}(F)\right);\xi_n^c(x-\Delta,{L_0},b)\right].
\end{multline}
As $|0\leq Z\leq Z_b \leq e^{V(\omega_n)}\ind{V(\omega_n) = M_n, \omega_n \in J_{\lambda, x-\Delta,K_0,L_0}(n)}$, by \eqref{eq:ajoutezeta} and Fact~\ref{fact:splitting} applied to $\eta=\frac{\varepsilon e^{\Delta}}{2(1+L_0)^2}$, this quantity is bounded from above by
\begin{equation}
\label{eqn:gamma}
  \left|\hat{\E}\left[Z_b\left( \exp\left(-e^{-\beta x}\mathring{\mu}^{L_0}_{n,n-b,\beta}(F)\right)- \exp\left(-e^{-\beta x}\mathring{\mu}^{L_0}_{n,n-b,\beta}(1)\times X_{F,\varepsilon_0}\right)\right); \xi_n(x-\Delta,L_0,b)\right]\right|+\varepsilon x e^{-x}.
\end{equation}

It remains to bound the first term of \eqref{eqn:gamma}. We compare $\mathring{\mu}^{L_0}_{n,n-b,\beta}(F)$ with $\mathring{\mu}^{L_0}_{n,n-b,\beta}(1)\times X_{F,\varepsilon_0}$, by comparing $F(\widetilde{H}^n_s(u); s\in[0,1])$ with $F(H^{(n)}_s(u); s\in[0,1])$.

Recall that $F$ is a continuous function on $D[0,1]$. Let us consider a relative compact $A\subset D[0,1]$ which means that 
\[
\max_{f\in A}\max_{0\leq x\leq 1}|f(x)|\leq M, \quad\limsup_{\delta\downarrow0}\sup_{f\in A}w(f,\delta)=0,
\]
where $M>0$ and $w(f,\delta):=\inf_{(t_i)}\max_{t_{i-1}\leq s\leq t<t_i} |f(s)-f(t)|$ with $0=t_0<t_1<\cdots<t_K=1$ and $\max(t_i-t_{i-1})\geq \delta$ is the continuous modulus of $D$. Then $F$ is uniformly continuous in $A$.

According to the weak convergence in $D[0,1]$ obtained in Lemma \ref{lem:rw}, for any $\varepsilon_0>0$, there exists a compact $K_f\subset D[0,1]$ such that for $n\geq n(f,K_f,\varepsilon_0)$,
\[
\sup_{y\in[0,r_n]}n^{3/2}\E\left[\ind{\ns\notin K_f} f(S_n-y); \mS_n\geq0, \mS_{[\lambda n, n]}\geq y\right]\leq \varepsilon_0.
\]
Similarly, by Lemma \ref{cor:rw}, one could find a compact $K'_{f}$ such that for all $a\in[0,r_n]$ and all $n$ sufficiently large,
\[
\sup_{y\in[0,r_n]}n^{3/2}\E\left[\ind{\ns\notin K'_{f}} f(S_n-y); \mS_n\geq-a, \mS_{[\lambda n, n]}\geq y\right]\leq \varepsilon_0 V^-(a).
\]

Now we take two compacts $\mathcal{K}_{1}=K'_{\ind{[0,2L_0]}}$ and $\mathcal{K}_2=K'_{G_{L_0,b}}$ with $G_{L_0,b}$ given below. Then $\mathcal{K}=\mathcal{K}_1\cup \mathcal{K}_2$ is also a compact. 
Since $F$ is uniformly continuous in the compact $\mathcal{K}$, for $\varepsilon_0>0$, there exists $\delta_0>0$ such that on the set $\{\max_{n-b\leq k\leq n}V(u_k)\leq \delta_0\sqrt{n}\}\cap\{\widetilde{H}^n(\omega_n)\in \mathcal{K}\}\cap\{H^n(u)\in \mathcal{K}\}$,
\[
\vert F(\widetilde{H}^n_s(u); s\in[0,1])-F(H^{(n)}_s(u); s\in[0,1])\vert\leq \varepsilon_0.
\]
And on the complement set, $\vert F(\widetilde{H}^n_s(u); s\in[0,1])-F(H^{(n)}_s(u); s\in[0,1])\vert\leq 1$ as $0\leq F\leq1$. One then observes that
\begin{align}\label{eq:changeF}
  &\left|\hat{\E}\left[Z_b\left( \exp\left(-e^{-\beta x}\mathring{\mu}^{L_0}_{n,n-b,\beta}(F)\right)- \exp\left(-e^{-\beta x}\mathring{\mu}^{L_0}_{n,n-b,\beta}(1)\times X_{F,\varepsilon_0}\right)\right); \xi_n(x-\Delta,L_0,b)\right]\right|\nonumber\\
  &\leq\hat{\E}\left[Z_b e^{-\beta x}\left\vert\mathring{\mu}^{L_0}_{n,n-b,\beta}(F)-\mathring{\mu}^{L_0}_{n,n-b,\beta}(1)\times X_{F,\varepsilon_0}\right\vert; \xi_n(x-\Delta,L_0,b)\right]\nonumber\\
  &\leq\Upsilon^\star_n+\Upsilon_{n,\mathcal{K},\omega_n}+\Upsilon_{n,\mathcal{K}}(\xi_n)
  \end{align}
  where 
  \begin{align*}
\Upsilon_{n,\mathcal{K}}(\xi_n):=&2\hat{\E}\left[Z_be^{-\beta x} n^{3\beta/2}\sum_{u\in J_{\lambda,x-\Delta,K_0,L_0}^{{L_0}}(n)}\ind{u\geq \omega_{n-b}}e^{-\beta V(u)}\ind{H^n(u)\notin \mathcal{K}}; \xi_n(x-\Delta,L_0,b)\right],\\
\Upsilon^\star_{n}:=&2\hat{\E}\left[Z_be^{-\beta x} n^{3\beta/2}\sum_{u\in J_{\lambda,x-\Delta,K_0,L_0}^{{L_0}}(n)}\ind{u\geq \omega_{n-b}}e^{-\beta V(u)}\left(\varepsilon_0+\ind{\max_{n-b\leq k\leq n}V(u_k)\geq \delta_0\sqrt{n}}\right)\right],\\
\Upsilon_{n,\mathcal{K},\omega_n}:=&2\hat{\E}\left[Z_be^{-\beta x} n^{3\beta/2}\sum_{u\in J_{\lambda,x-\Delta,K_0,L_0}^{{L_0}}(n)}\ind{u\geq \omega_{n-b}}e^{-\beta V(u)}\ind{\widetilde{H}^n(\omega_n)\notin\mathcal{K}}\right]
\end{align*}
Note that $\Upsilon^\star_n$ by the Markov property at time $n-b$ is equal to
\begin{multline}
2n^{3/2}e^{-x+(\beta-1)({L_0}-\Delta)}\hat{\E}\Big[G_{{L_0},b,\varepsilon_0}\left(V(\omega_{n-b})-a_n(z+L_0)\right);\\
\min_{0\leq k\leq n-b}V(\omega_k)\geq-z+K_0, \min_{\lambda n \leq k\leq n-b}V(\omega_k)\geq a_n(z+L_0)\Big]
%\hat{\E}\Big[G_{{L_0},b}\left(V(\omega_{n-b})-a_n(z+L_0)\right); H^n(\omega_{n-k})\notin \mathcal{K},
%\min_{0\leq k\leq n-b}V(\omega_k)\geq-z+K_0, \min_{\lambda n \leq k\leq n-b}V(\omega_k)\geq a_n(z+L_0)\Big]\bigg\},
\end{multline}
where $G_{L_0,b,\varepsilon_0}(t)$ is defined as
\begin{multline}
\hat{\E}_t\Bigg[\frac{e^{V(\omega_b)}\ind{V(\omega_b)=M_b;\min_{0\leq k\leq b}V(\omega_k)\geq 0, V(\omega_b)\leq {L_0}}}{\sum_{|u|=b}\ind{V(u)=M_b}}\sum_{|u|=b}e^{-\beta V(u)} \ind{\min_{0\leq k\leq b}V(u_k)\geq0, V(u)\leq 2{L_0}}\times\\
\left(\varepsilon_0+\ind{\max_{0\leq k\leq b}V(u_k)\geq \delta_0\sqrt{n}}\right)\Bigg].
\end{multline}
To bound $G_{{L_0},b,\varepsilon_0}(t)$, we return to the probability $\P$ and observe that
\begin{multline*}
  G_{{L_0},b,\varepsilon_0}(t)
  = e^{(1-\beta) t}\E\Bigg[\sum_{|u|=b} e^{-\beta V(u)}\ind{\min_{0\leq k\leq b}V(u_k)\geq-t, V(u)\leq 2{L_0}-t} \\
\times\left(\varepsilon_0+\ind{\max_{0\leq k\leq b}V(u_k)\geq \delta_0\sqrt{n}}\right)\ind{\min_{0\leq k\leq b}V(m_k^{(b)})\geq -t, V(m^{(b)})\leq {L_0}-t}\Bigg],
\end{multline*}
which is bounded by
\begin{equation*}
 e^{(1-\beta) t}\E\left[\sum_{|u|=b}e^{-\beta V(u)}\ind{\min_{0\leq k\leq b}V(u_k)\geq-t, V(u)\leq 2{L_0}-t}\left(\varepsilon_0+\ind{\max_{0\leq k\leq b}V(u_k)\geq \delta_0\sqrt{n}}\right)\right].
\end{equation*}
By Many-to-one lemma,
\begin{align*}
  G_{{L_0},b,\varepsilon_0}(t)
  &\leq \E_t\left(e^{(1-\beta)S_b}(\varepsilon_0+\ind{\max_{0\leq k\leq b}S_k\geq \delta_0\sqrt{n}}); S_b\leq 2{L_0}, \mS_b\geq0\right)\\
  &\leq 2\P_t(S_b\leq 2{L_0}) \leq 2\P(2{L_0}-S_b\geq t).
\end{align*}
We observe that the function $t \mapsto \P(2{L_0}-S_b \geq t)$ is non-increasing, and
\[ \int_0^{\infty} t \P(2{L_0}-S_b \geq t) dt \leq \frac{1}{2}\E((2{L_0}-S_b)^2) < \infty. \]
Using the dominated convergence theorem, we have
\begin{equation}
  \lim_{n\to \infty} \int_{\R_+}G_{{L_0},b,\varepsilon_0}(t)tdt \leq \frac{1}{2}\E((2{L_0}-S_b)^2)\varepsilon_0.
\end{equation}

Moreover, the function $G_{{L_0},b,\varepsilon_0}$ is Riemann-integrable. Therefore, using Lemma~\ref{lem:rw} proves that for  all $n$ sufficiently large,
\begin{equation*}
\Upsilon^\star_n
\leq c_{16}\varepsilon_0xe^{-x+(\beta-1)(L_0-\Delta)}.
\end{equation*}
Similarly, one sees that $\Upsilon_{n,K,\omega_n}$ is equal to
\begin{multline*}
2n^{3/2}e^{-x+(\beta-1)({L_0}-\Delta)}\hat{\E}\Big[G_{{L_0},b}\left(V(\omega_{n-b})-a_n(z+L_0)\right);\\
H^n(\omega_{n-k})\notin \mathcal{K},
\min_{0\leq k\leq n-b}V(\omega_k)\geq-z+K_0, \min_{\lambda n \leq k\leq n-b}V(\omega_k)\geq a_n(z+L_0)\Big],
\end{multline*}
where 
$G_{L_0,b}(t)$ as
\[
 \hat{\E}_t\Bigg(\frac{e^{V(\omega_b)}\ind{V(\omega_b)=M_b;\min_{0\leq k\leq b}V(\omega_k)\geq 0, V(\omega_b)\leq {L_0}}}{\sum_{|u|=b}\ind{V(u)=M_b}}
 \sum_{|u|=b}e^{-\beta V(u)} \ind{\min_{0\leq k\leq b}V(u_k)\geq0, V(u)\leq 2{L_0}}\Bigg).
\]
Similarly to $G_{L_0,b,\varepsilon_0}$, $G_{L_0,b}$ is Riemann-integrable and $\int_{\R_+}G_{{L_0},b}(t)tdt \leq \frac{1}{2}\E((2{L_0}-S_b)^2)$. So Lemma \ref{cor:rw} can be applied and we could find the compact $K'_{G_{L_0,b}}\subset \mathcal{K}$ as desired. Consequently, for all $n$ sufficiently large,
\[
\Upsilon_{n,\mathcal{K}, \omega_n}\leq 2\varepsilon_0 V^-(x)e^{-x+(\beta-1)(L_0-\Delta)}. 
\]
It remains to bound $\Upsilon_{n,\mathcal{K}}(\xi_n)$. Observe that by change of measures,
\begin{align*}
&\Upsilon_{n,\mathcal{K}}(\xi_n)\leq 2\hat{\E}\left[Ze^{-\beta x} n^{3\beta/2}\sum_{|u|=n}\ind{u\in J_{\lambda,x-\Delta,K_0,L_0}^{{L_0}}(n)}e^{-\beta V(u)}\ind{H^n(u)\notin \mathcal{K}}\right]\\
=&2 e^{-\beta x}n^{3\beta/2}\E\left[\sum_{|u|=n}\ind{u\in J_{\lambda,x-\Delta,K_0,L_0}^{{L_0}}(n)}e^{-\beta V(u)}\ind{H^n(u)\notin \mathcal{K}}; E_n\right]\\
\leq & 2 e^{-\beta x}n^{3\beta/2}\E\left[e^{(1-\beta)S_n}; \begin{array}{l} S_n\leq a_n(x-\Delta-L_0), \mS_n\geq -x+\Delta+K_0,\\ \mS_{[\lambda n, n]}\geq a_n(x-\Delta+L_0), \ns\notin\mathcal{K}\end{array}\right]\\
\leq & 2 e^{-x+(\beta-1)(L_0-\Delta)}\E\left[\ind{\ns\notin\mathcal{K}}; S_n-y\in [0, 2L_0], \mS_n\geq -x+\Delta+K_0, \mS_{[\lambda n, n]}\geq y \right]
\end{align*}
with $y=a_n(x-\Delta+L_0)$. As we choose $K'_{\ind{[0,2L_0]}}\subset \mathcal{K}$, it follows that
\[
\Upsilon_{n,\mathcal{K}}(\xi_n)\leq 2\varepsilon_0 V^-(x)e^{-x+(\beta-1)(L_0-\Delta)}.
\]
Applying these estimates to \eqref{eq:changeF}, for all $n$ large enough, we have
\begin{multline*}
\left|\hat{E}\left[Z_b\left( \exp\left(-e^{-\beta x}\mathring{\mu}^{L_0}_{n,n-b,\beta}(F)\right)- \exp\left(-e^{-\beta x}\mathring{\mu}^{L_0}_{n,n-b,\beta}(1)\times X_{F,\varepsilon_0}\right)\right); \xi_n(x-\Delta, L_0, b)\right]\right|\\
\leq 2c_{16}\varepsilon_0e^{(\beta-1)(L_0-\Delta)}xe^{-x}.
\end{multline*}
In view of \eqref{eq:ChangeF} and \eqref{eqn:gamma}, we can choose $\varepsilon_0>0$ sufficiently small so that 
\begin{equation}
\label{eq:controlF}
\left|\widehat{\Upsilon}^{L_0}_{E_n}(n,n-b,x,F)-\hat{\E}\left[Z_b \exp\left(-e^{-\beta x}\mathring{\mu}^{L_0}_{n,n-b,\beta}(1)\times X_{F,\varepsilon_0}\right)\right]\right|\leq 2\varepsilon xe^{-x}.
\end{equation}

In the similar way, we get that 
\begin{equation}
\label{eq:controlFe}
\left|\widehat{\Upsilon}^{L_0}_{E_n}(n,n-b,x,F(\epsilon))-\hat{\E}\left[Z_b\exp\Big(-e^{-\beta x}\mathring{\mu}^{L_0}_{n,n-b,\beta}(1)\times\Big(F(\epsilon)\vee\varepsilon_0\Big) \Big)\right]\right|\leq 2\varepsilon xe^{-x}.
\end{equation}

We now consider the quantity $\frac{e^x}{x}\hat{\E}\left[Z_b \exp\left(-e^{-\beta x}\mathring{\mu}^{L_0}_{n,n-b,\beta}(1)\times X_{F,\varepsilon_0}\right)\right]$ and show that it is close to $\frac{e^x}{x}\hat{\E}\left[Z_b\exp\left(-e^{-\beta x}\mathring{\mu}^{L_0}_{n,n-b,\beta}(1)\times\left(F(\epsilon)\vee\varepsilon_0\right) \right)\right]$. However, we can not compare these quantities directly, thus we prove that
\begin{multline*} \frac{e^x}{x}\hat{\E}\left[Z_b\left(1-\exp\left(-e^{-\beta x}\mathring{\mu}^{L_0}_{n,n-b,\beta}(1)\times X_{F,\varepsilon_0} \right)\right)\right]\\ \sim_{n \to \infty}
 \frac{e^x}{x}\hat{\E}\left[Z_b\left(1-\exp\left(-e^{-\beta x}\mathring{\mu}^{L_0}_{n,n-b,\beta}(1)\times \left(F(\epsilon)\vee\varepsilon_0\right) \right)\right)\right].
 \end{multline*}

Applying the equation
\[1-e^{-\lambda W}=\int_{\r_+}\lambda e^{-\lambda t}\ind{W\geq t}dt\]
with $\lambda=X_{F,\varepsilon_0}$ leads to 
\begin{multline}\label{eq:tireF}
  \frac{e^x}{x}\hat{\E}\left[Z_b\left(1-\exp\left(-e^{-\beta x}X_{F,\varepsilon_0}\mathring{\mu}^{L_0}_{n,n-b,\beta}(1) \right)\right)\right]
  = \int_{\r_+}\frac{e^x}{x}\hat{\E}\left[Z_b X_{F,\varepsilon_0}e^{-tX_{F,\varepsilon_0}}; \mathring{\mu}^{L_0}_{n,n-b,\beta}(1)\geq te^{\beta x}\right]dt \\
  =\int_{\r}\beta\frac{e^x}{x}\hat{\E}\left[Z_b e^{\beta y}X_{F,\varepsilon_0}e^{-e^{\beta y}X_{F,\varepsilon_0}}; \mathring{\mu}^{L_0}_{n,n-b,\beta}(1)\geq e^{\beta (x+y)}\right]dy,
\end{multline}
by change of variables $t=e^{\beta y}$. Applying the Markov property at time $n-b$ implies that
\begin{multline}
  \label{eqn:342}
  \frac{e^x}{x} \hat{\E}\left[ Z_b e^{\beta y} X_{F,\varepsilon_0} e^{-e^{\beta y} X_{F,\varepsilon_0}}; \mathring{\mu}^{L_0}_{n,n-b,\beta}(1)\geq e^{\beta (x+y)}\right]\\
  =n^{3/2} \frac{e^{\Delta}}{x}\hat{\E}\Bigg[e^{\beta y}X_{F,\varepsilon_0}e^{-e^{\beta y}X_{F,\varepsilon_0}}f_{{L_0},b}\left(V(\omega_{n-b})-a_n(z+{L_0}), y+\Delta\right);\\
\min_{0\leq k\leq n-b}V(\omega_k)\geq-z+K_0, \min_{\lambda n \leq k\leq n-b}\geq a_n(z+{L_0})\Bigg],
\end{multline}
where $z=x-\Delta$ and
\begin{multline}
  f_{L_0,b}(z,y):=\hat{\E}_z\Bigg[\frac{e^{V(\omega_b)-{L_0}}\ind{V(\omega_b)=M_b}}{\sum_{|u|=b}\ind{V(u)=M_b}}\ind{\min_{0\leq k\leq b}V(\omega_k)\geq0, V(\omega_b)\leq {L_0}}\\
\times \ind{\sum_{|u|=b}e^{-\beta V(u)} \ind{\min_{0\leq k\leq b}V(u_k)\geq0, V(u)\leq 2{L_0}}\geq e^{\beta(y-{L_0})}}\Bigg].
\end{multline}

According to Lemma 5.4 in~\cite{Mad11}, $f_{{L_0},b}$ is Riemann integrable and bounded by $\P(S_b\leq {L_0}-z)$. For all $y\in\r_+$ and $n\geq 10 b$, we have
\begin{equation}
\frac{e^x}{x}\hat{\E}\left[Z_b e^{\beta y}X_{F,\varepsilon_0}e^{-e^{\beta y}X_{F,\varepsilon_0}}; \mathring{\mu}^{L_0}_{n,n-b,\beta}(1)\geq e^{\beta (x+y)}\right]\leq c_{17} e^{\beta y}e^{-\varepsilon_0 e^{\beta y}},
\end{equation}
which is integrable with respect to the Lebesgue measure. Using \eqref{eqn:342}, then using Proposition~\ref{prop:spinaldecomposition} to identify $(V(\omega_k), k\geq 0)$ with a random walk and applying Lemma~\ref{lem:rw}, we obtain that for any $y\in\r$, as $n\to\infty$,
\begin{multline*}
  \lim_{n\rightarrow\infty}\frac{e^x}{x}\hat{\E}\left[Z_b e^{\beta y}X_{F,\varepsilon_0}e^{-e^{\beta y}X_{F,\varepsilon_0}}; \mathring{\mu}^{L_0}_{n,n-b,\beta}(1)\geq e^{\beta (x+y)}\right]\\
  = C_1 \frac{e^{\Delta} V^-(z-K_0)}{x}\int_{\r_+}f_{{L_0},b}(z,y+\Delta)V^+(z)dz \E\left[e^{\beta y}\left(F(\epsilon)\vee\varepsilon_0\right)e^{-e^{\beta y}\left(F(\epsilon)\vee\varepsilon_0\right)}\right].
\end{multline*}
By the exact same arguments, we have
\begin{multline*}
  C_1 \frac{e^{\Delta} V^-(z-K_0)}{x}\int_{\r_+}f_{{L_0},b}(z,y+\Delta)V^+(z)dz \E\left[e^{\beta y}\left(F(\epsilon)\vee\varepsilon_0\right)e^{-e^{\beta y}\left(F(\epsilon)\vee\varepsilon_0\right)}\right]\\
  =\lim_{n\rightarrow\infty}\frac{e^x}{x}\hat{\E}\left[Z_b e^{\beta y}\left(F(\epsilon)\vee\varepsilon_0\right)e^{-e^{\beta y}\left(F(\epsilon)\vee\varepsilon_0\right)}; \mathring{\mu}^{L_0}_{n,n-b,\beta}(1)\geq e^{\beta (x+y)}\right].
\end{multline*}

Therefore, applying the dominated convergence theorem, \eqref{eq:tireF} becomes
\begin{align}
  &\frac{e^x}{x}\hat{\E}\left[Z_b\left(1-\exp\left(-e^{-\beta x}X_{F,\varepsilon_0}\mathring{\mu}^{L_0}_{n,n-b,\beta}(1) \right)\right)\right]\nonumber\\
  &\qquad \qquad =\int_{\r}\beta\frac{e^x}{x}\hat{\E}\left[Z_b e^{\beta y}X_{F,\varepsilon_0}e^{-e^{\beta y}X_{F,\varepsilon_0}}; \mathring{\mu}^{L_0}_{n,n-b,\beta}(1)\geq e^{\beta (x+y)}\right]dy\nonumber\\
  &\qquad \qquad =\int_{\r}\beta\frac{e^x}{x}\hat{\E}\left[Z_b e^{\beta y}\left(F(\epsilon)\vee\varepsilon_0\right)e^{-e^{\beta y}\left(F(\epsilon)\vee\varepsilon_0\right)}; \mathring{\mu}^{L_0}_{n,n-b,\beta}(1)\geq e^{\beta (x+y)}\right]dy+o_n(1)\nonumber\\
  &\qquad \qquad =\frac{e^x}{x}\hat{\E}\left[Z_b\left(1-\exp\left(-e^{-\beta x}\left(F(\epsilon)\vee\varepsilon_0\right)\mathring{\mu}^{L_0}_{n,n-b,\beta}(1) \right)\right)\right]+o_n(1).
\label{eq:convergenceF}
\end{align}

Thus, we obtain that for all $n$ sufficiently large,
\begin{multline*}
\bigg|\frac{e^x}{x}\hat{\E}\left[Z_b \exp\left(-e^{-\beta x}\mathring{\mu}^{L_0}_{n,n-b,\beta}(1)\times X_{F,\varepsilon_0}\right)\right]\\
-\frac{e^x}{x}\hat{\E}\left[Z_b\exp\Big(-e^{-\beta x}\mathring{\mu}^{L_0}_{n,n-b,\beta}(1)\times\left(F(\epsilon)\vee\varepsilon_0\right) \Big)\right]\bigg|\leq \varepsilon.
\end{multline*}

In view of \eqref{eq:controlF} and \eqref{eq:controlFe}, we checked that for all $n$ sufficiently large,
\begin{equation}
\left|\widehat{\Upsilon}^{L_0}_{E_n}(n,n-b,x,F)-\widehat{\Upsilon}^{L_0}_{E_n}(n,n-b,x,F(\epsilon))\right|\leq 5\varepsilon xe^{-x}.
\end{equation}

It hence follows that
for all $n$ sufficiently large,
\begin{equation}
\left|\widehat{\Sigma}^{L_0}_{E_n}(n,n-b,x,F)-\E\left[\exp\left(-e^{-\beta x}\widehat{\mu}^{L_0}_{n,n-b,\beta}(1) F(\epsilon)\mathbf{1}_{E_n}\right)\right]\right|\leq 5\varepsilon xe^{-x}.
\end{equation}

It remains to compare $\E\left[\exp\left(-e^{-\beta x}\widehat{\mu}^{L_0}_{n,n-b,\beta}(1) F(\epsilon)\mathbf{1}_{E_n}\right)\right]$ with $\E\left[\exp\left(-e^{-\beta x}\widetilde{\mu}_{n,\beta} F(\epsilon)\right)\right]$ (recall that $\tilde{\mu}_{n,\beta}$ was defined in \eqref{eqn:defTildeMu}).

Applying Lemmas~\ref{lem:ajouteE},~\ref{lem:ajouteL} and~\ref{lem:compteB} to $\E\left(\exp\{-e^{-\beta x}\tilde{\mu}_{n,\beta}(1) F(\epsilon)\}\right)$ implies that
\begin{equation}
  0\leq\E\left[\exp\left(-e^{-\beta x}\widehat{\mu}^{L_0}_{n,n-b,\beta}(1) F(\epsilon)1_{E_n}\right)\right]-\E\left[\exp\left(-e^{-\beta x}\widetilde{\mu}_{n,\beta}(1) F(\epsilon)\right)\right]\leq 3\varepsilon x e^{-x}.
\end{equation}

As a consequence, for all $n$ sufficiently large,
\begin{equation}
\left|\widehat{\Sigma}^{L_0}_{E_n}(n,n-b,x,F)-\E\left[\exp\left(-e^{-\beta x}\widetilde{\mu}_{n,\beta} F(\epsilon)\right)\right]\right|\leq 8 \varepsilon xe^{-x},
\end{equation}
which completes the proof.
\end{proof}

We now prove Proposition~\ref{tailprop}.

\begin{proof}[Proof of Proposition~\ref{tailprop}]
For any non-negative $ F\in \mathcal{C}_b(\mathcal{D})$ and $\varepsilon>0$, we choose $\Delta=\Delta_\varepsilon:=\Delta_{\varepsilon,1}\vee \Delta_{\varepsilon,2}$ and $L=L_0=\frac{2\Delta}{\alpha_1\wedge\alpha_2}$. Set $K_0=K_{\varepsilon,L_0}+L_0$, $n\geq n_\varepsilon$ and $A\geq 2e^{K_0+\Delta}/\varepsilon$, we observe that
\begin{multline*}
  \left|\Sigma(n,x,F)-\E\left[\exp\left(-e^{-\beta x}\tilde{\mu}_{n,\beta}(1) F(\epsilon)\right)\right]\right|\\
  \leq \left|\Sigma(n,x,F)-\Sigma_{E_n}(n,x,F)\right|
  +\left|\Sigma^{L_0}_{E_n}(n,x,F)-\Sigma_{E_n}(n,x,F)\right|\\
  +\left|\widehat{\Sigma}^{L_0}_{E_n}(n,n-b,x,F)-\Sigma^{L_0}_{E_n}(n,x,F)\right|\\
  +\left|\widehat{\Sigma}^{L_0}_{E_n}(n,n-b,x,F)-\E\left(\exp\left(-e^{-\beta x}\tilde{\mu}_{n,\beta}(1) F(\epsilon)\right)\right]\right|.
\end{multline*}
Using Lemmas~\ref{lem:ajouteE},~\ref{lem:ajouteL},~\ref{lem:compteB} and~\ref{lem:indep}, we have
\begin{equation}
\left|\Sigma(n,x,F)-\E\left[\exp\left(-e^{-\beta x}\tilde{\mu}_{n,\beta}(1) F(\epsilon)\right)\right]\right|\leq 4 \varepsilon x e^{-x},
\end{equation}
where $\tilde{\mu}_{n,\beta}(1)$ and $F(\epsilon)$ are independent. Recall that $\Sigma(n,x,F)=\E\left[\exp\left(-e^{-\beta x}\widetilde{\mu}_{n,\beta}(F)\right)\right]$.  It hence follows that
\begin{multline}
\label{eq:facteurFe}
  \left|\frac{e^x}{x}\E\left[1-\exp\left(-  \theta e^{-\beta x} \widetilde{\mu}_{n,\beta}(F)\right)\right]-\frac{e^x}{x}\E\left[1-\exp\left(-  \theta e^{-\beta x} \tilde{\mu}_{n,\beta}(1) F(\epsilon)\right)\right]\right|\\
\leq \frac{e^x}{x}\left|\Sigma(n,x,F)-\E\left[\exp\left(-e^{-\beta x}\tilde{\mu}_{n,\beta}(1) F(\epsilon)\right)\right]\right|\leq 4 \varepsilon.
\end{multline}
We replace $\theta$ by $\theta F(\epsilon)$, and then deduce from \eqref{sansF} that for all $n$ sufficiently large,
\begin{equation*}
\left|\frac{e^x}{x}\E\left[\left.1-\exp\left(-  \theta e^{-\beta x} \tilde{\mu}_{n,\beta}(1) F(\epsilon)\right) \right| \epsilon\right]-C_\beta \theta^{\frac{1}{\beta}} F(\epsilon)^{\frac{1}{\beta}}\right|\leq \varepsilon. 
\end{equation*}
In particular, for all $n$ sufficiently large,
\begin{equation*}
\left|\frac{e^x}{x}\E\left[1-\exp\left(-  \theta e^{-\beta x} \tilde{\mu}_{n,\beta}(1) F(\epsilon)\right)\right]-C_\beta \theta^{\frac{1}{\beta}} \E\left[F(\epsilon)^{\frac{1}{\beta}}\right]\right|\leq \varepsilon.
\end{equation*}
Going back to \eqref{eq:facteurFe}, we have
\begin{equation*}
\left|\frac{e^x}{x}\E\left(1-\exp\{-  \theta e^{-\beta x} \widetilde{\mu}_{n,\beta}(F)\}\right)-C_\beta \theta^{\frac{1}{\beta}} \E[F(\epsilon)^{\frac{1}{\beta}}]\right|\leq 5\varepsilon,
\end{equation*}
which completes the proof and gives Proposition~\ref{tailprop}.
\end{proof}

\section{Proof of Theorem~\ref{thm:main} and Corollary~\ref{cor:interest}}
\label{sec:conclusion}

We apply the Laplace transform estimates obtained in the previous section to prove the main results of this article. We first study the convergence of the Laplace transform of $\tilde{\mu}_{n,\beta}(F)$. We recall that $Z_n = \sum_{|u|=n} V(u)e^{-V(u)}$ is a martingale, and that $Z_\infty = \lim_{n \to \infty} Z_n$ a.s.

\begin{proposition}
\label{propfi}
Under \eqref{eqn:supercritical}, \eqref{eqn:boundary}, \eqref{eqn:variance} and \eqref{eqn:integrability}, for any $\alpha \geq 0$ and any non-negative $F\in \calC_b(\calD)$,
\begin{multline}
\label{finnn}
  \lim_{l \to \infty} \limsup_{n \to \infty} \E\left[ \ind{Z_l>0} e^{-\alpha Z_l-\tilde{\mu}_{n+l,\beta}(F)} \right] = \lim_{l \to \infty} \liminf_{n \to \infty} \E\left[ \ind{Z_l>0}e^{-\alpha Z_l -\tilde{\mu}_{n+l,\beta}(F)} \right]\\
  = \E\left[ \exp\left(-C_\beta Z_\infty \E\left( F(\epsilon)^\frac{1}{\beta} \right)\right) e^{-\alpha Z_\infty} \ind{Z_\infty > 0}\right].
\end{multline}
In particular, conditionally on the survival event $S$, we have
\begin{equation}
\label{eqn:finnnn}
 \lim_{n \to \infty} \E\left[ \left.e^{-\tilde{\mu}_{n,\beta}(F)}\right| S \right] = \E\left[ \left.\exp\left(-C_\beta Z_\infty \E\left( F(\epsilon)^\frac{1}{\beta} \right)\right)\right|S \right].
\end{equation}
\end{proposition}

\begin{remark}
Theorem~\ref{thm:main} is a consequence of \eqref{eqn:finnnn}. However \eqref{finnn} enlightens the appearance of $Z_\infty$.
\end{remark}

\begin{proof}
Note that \eqref{eqn:finnnn} is a direct consequence of \eqref{finnn} as $S = \{Z_\infty>0\}$. We observe that
\begin{equation}
  \tilde{\mu}_{n+l,\beta}(F) = \left( \frac{n+l}{n}\right)^\frac{3\beta}{2} \sum_{|u|=l} e^{-\beta V(u)} n^{\frac{3\beta}{2}} \sum_{\substack{|v|=n+l \\ v \geq u}} e^{-\beta (V(v)-V(u))} F(H^{n+l}(v)).
\end{equation}
For $|u|=l$, $v\geq u$ with $|v|=n+l$ and $t \in [0,1]$, we write $H^{(n),u}(v)_t := \frac{V(v_{l+\floor{nt}})-V(u)}{\sqrt{\sigma^2 n}}$ and
\[
  \tilde{\mu}^{(u)}_{n,\beta}(F) := n^{3\beta/2}\sum_{\substack{|v|=n+l \\ v \geq u}} e^{-\beta (V(v)-V(u))} F\left(H^{(n),u}(v)\right).
\]
If we compare $ \tilde{\mu}_{n+l,\beta}(F)$ with $ \sum_{|u|=l} e^{-\beta V(u)} \tilde{\mu}^{(u)}_{n,\beta}(F)$, we obtain that for any compact $\mathcal{K}_0$ and for all sufficiently large $n$ and $\log n \gg \ell$,
\begin{align*}
&\chi_n:=\Big\vert \tilde{\mu}_{n+l,\beta}(F)-\sum_{|u|=l} e^{-\beta V(u)} \tilde{\mu}^{(u)}_{n,\beta}(F)\Big\vert\\
\leq & o_n(1) \tilde{\mu}_{n+\ell,\beta}+ \sum_{|u|=\ell}e^{-\beta V(u)}\ind{\max_{1\leq k\leq \ell}V(u_k)\geq \log n}\tilde{\mu}^{(u)}_{n,\beta}(1)\\
&+ \tilde{\mu}_{n+\ell,\beta}(\ind{H^{(n+\ell)}\notin\mathcal{K}_0}+\ind{H^{(n),v_\ell}\notin \mathcal{K}_0; \max_{1\leq k\leq \ell}V(u_k)\leq \log n})
\end{align*}

As long as we could prove the convergence in probability of $\chi_n$ towards zero, the convergence in law of $ \sum_{|u|=l} e^{-\beta V(u)} \tilde{\mu}^{(u)}_{n,\beta}(F)$ holds also for $ \tilde{\mu}_{n+l,\beta}(F)$. In fact, by Proposition 4.6 of \cite{Mad11}, $\tilde{\mu}_{n+\ell,\beta}$ is tight. Moreover, one sees that $\P(\max_{|u|=\ell}\max_{1\leq k\leq \ell}V(u_k)\geq \log n)\rightarrow 0$ as $n\uparrow\infty$. So we only need to consider $\tilde{\mu}_{n,\beta}(\ind{H^{(n)}\notin\mathcal{K}_0})$ for some well chosen compact $\mathcal{K}_0$.  Observe that for any $\delta>0$,
\begin{align*}
&\P\left(\tilde{\mu}_{n,\beta}\ind{H^{(n)}\notin\mathcal{K}_0}\geq \delta\right)\\
\leq &\P(\inf V(u)\leq -y)+\P\left(n^{3\beta/2}\sum_{|u|=n}e^{-\beta V(u)}\ind{\min_{1\leq k\leq n}V(u_k)\geq-y; H^{(n)}(u)\notin\mathcal{K}_0}\geq \delta\right)\\
\leq & e^{-y}+\P\left(n^{3\beta/2}\sum_{|u|=n}e^{-\beta V(u)}\ind{\min_{1\leq k\leq n}V(u_k)\geq-y; H^{(n)}(u)\notin\mathcal{K}_0}\geq \delta\right)
\end{align*}
We are going to take $y\geq1$ large enough so that the first  term is sufficiently small. For $C\geq3/2$ a large constant and $L\geq0$, $x=\frac{\log \delta}{\beta}$, one sees that 
\begin{equation}\label{bdmuKc}
\P\left(n^{3\beta/2}\sum_{|u|=n}e^{-\beta V(u)}\ind{\min_{1\leq k\leq n}V(u_k)\geq-y; H^{(n)}(u)\notin\mathcal{K}_0}\geq 2\delta\right)\leq \P_1+\P_2
\end{equation}
where
\begin{align*}
\P_1:=&\P\left(n^{3\beta/2}\sum_{|u|=n}e^{-\beta V(u)}\ind{\min_{1\leq k\leq n}V(u_k)\geq-y, \min_{n/2\leq k\leq n}V(u_k)\leq a_n(x+L)}\geq e^{\beta x}\right)\\
\P_2:=&\P\left(n^{3\beta/2}\sum_{|u|=n}e^{-\beta V(u)}\ind{\min_{1\leq k\leq n}V(u_k)\geq-y, \min_{n/2\leq k\leq n}V(u_k)\geq a_n(x+L); H^{(n)}(u)\notin\mathcal{K}_0}\geq e^{\beta x}\right)
\end{align*}
Observe that 
\begin{align*}
\P_1=&\P_y\left(n^{3\beta/2}\sum_{|u|=n}e^{-\beta V(u)}\ind{\min_{1\leq k\leq n}V(u_k)\geq 0, \min_{n/2\leq k\leq n}V(u_k)\leq a_n(x-y+L)}\geq e^{\beta (x-y)}\right)
\end{align*}
which by Lemma 4.9 of \cite{Mad11} is bounded by $C_1(1+y)e^{-C_2L-x}$ with $C_1, C_2>0$. Let us take $L=y$. Then by Markov inequality and Many-to-one lemma, one sees that
\begin{align*}
\P_2\leq & n^{3\beta/2}e^{-\beta x} \E\left[ e^{(1-\beta)S_n}; \mS_n\geq-y, \mS_{[n/2,n]}\geq a_n(x+y); \ns\notin \mathcal{K}_0\right]\\
=& e^{-x+(\beta-1)y} n^{3/2}\E\left[e^{(1-\beta)(S_n-a_n(x+y))}; \mS_n\geq-y, \mS_{[n/2,n]}\geq a_n(x+y); \ns\notin \mathcal{K}_0\right]
\end{align*}
By Lemma \ref{cor:rw}, for any $\eta>0$, we could find a compact $\mathcal{K}_0=K_{f}$ for $f(t)=e^{(1-\beta)t}$ which is Riemann integrable such that 
\[
\sup_{y\in[0,a_n]}n^{3/2}\E\left[e^{(1-\beta)(S_n-a_n(x+y))}; \mS_n\geq-y, \mS_{[n/2,n]}\geq a_n(x+y); \ns\notin \mathcal{K}_0\right]\leq \eta (1+y).
\]
Then given $\delta>0$, for any $\varepsilon>0$, we take $y=y_\varepsilon\geq1$ such that $e^{-y}+C_1(1+y)e^{-C_2y-x}\leq \varepsilon/2$, and we take the compact $\mathcal{K}_0$ for $\eta=\frac{\varepsilon}{2(1+y_\varepsilon)}e^{-(\beta-1)y_\varepsilon}\delta^{1/\beta}$. These choices imply that
\[
\P\left(\tilde{\mu}_{n,\beta}\ind{H^{(n)}\notin\mathcal{K}_0}\geq 2\delta\right)\leq \varepsilon.
\]
Similar arguments work for $\tilde{\mu}_{n+\ell,\beta}\ind{H^{(n),v_\ell}\notin \mathcal{K}_0; \max_{1\leq k\leq \ell}V(u_k)\leq \log n}$. We hence conclude the convergence in probability of $\chi_n$ to zero.

It remains to prove the convergence in law for $ \sum_{|u|=l} e^{-\beta V(u)} \tilde{\mu}^{(u)}_{n,\beta}(F)$. Applying the Markov property at time $l$, we have
\[
  \E\left[ e^{-\alpha Z_l}\ind{Z_l>0} e^{-\sum_{|u|=l} e^{-\beta V(u)} \tilde{\mu}^{(u)}_{n,\beta}(F)} \right] 
  = \E\left[ e^{-\alpha Z_l}\ind{Z_l>0}  \prod_{|u|=l} \Psi(V(u))\right],
\]
where $\Psi :x \mapsto \E\left[ e^{- e^{-\beta x} \tilde{\mu}_{n,\beta}(F)} \right]$. For any $1\leq l\leq n$, we set $\Xi_{n,l} := \left\{ |u| = l : \frac{\log l}{3} \leq V(u)\leq \log n \right\}$, and we prove that $\prod_{|u|=l} \Psi(V(u)) \approx \prod_{u \in \Xi_{n,l}} \Psi(V(u))$ with high probability. Note that
\begin{equation}
  \label{eqn:clear}
  \lim_{n \to \infty} \sum_{u \in \Xi_{n,l}} V(u) e^{-V(u)} = Z_l \text{ a.s. for $l$ large enough},
\end{equation}
as $\liminf_{l \to \infty} \frac{M_l}{\log l} = \frac{1}{2} > \frac{1}{3}$ (see~\cite{HuS09}).

We first observe that
\[
  \E\left[ e^{-\alpha Z_l}\ind{Z_l>0} e^{-\sum_{|u|=l} e^{-\beta V(u)} \tilde{\mu}^{(u)}_{n,\beta}(F)} \right] \leq \E\left[ e^{-\alpha Z_l}\ind{Z_l>0}  \prod_{u\in\Xi_{n,l}} \Psi(V(u))\right].
\]
By Proposition~\ref{tailprop}, for any $\varepsilon>0$, there exist $L,N$ such that for any $n \geq N$, $l \geq L$ and $u \in \Xi_{n,l}$,
\begin{equation}
 \label{constailprop} 
  \left| \Psi(V(u)) - 1 + \psi_\beta V(u)e^{-V(u)} \right| \leq \varepsilon V(u) e^{-V(u)}n,
\end{equation}
where $\psi_\beta = C_\beta \E\left[ F(\epsilon)^\frac{1}{\beta}\right]$. This yields
\begin{align*}
  \E\left[ e^{-\alpha Z_l}\ind{Z_l>0} e^{-\sum_{|u|=l} e^{-\beta V(u)} \tilde{\mu}^{(u)}_{n,\beta}(F)} \right]
  &\leq \E\left[e^{-\alpha Z_l} \ind{Z_l>0} \prod_{u\in  \Xi_{n,l}} \left( 1 - \left( \psi_\beta - \varepsilon\right) V(u)e^{-V(u)} \right) \right]\\
  &\leq \E\left[ e^{-\alpha Z_l}\ind{Z_l>0}  e^{- \sum_{u\in \Xi_{n,l}} V(u)e^{-V(u)}\left(\psi_\beta-\varepsilon\right)} \right],
\end{align*}
where the last inequality follows from the fact that for any $x \geq 0$, $1-x\leq e^{-x}$. By \eqref{eqn:clear} and the convergence of the derivative martingale, we have
\[
  \limsup_{n \to \infty} \E\left[ e^{-\alpha Z_l}\ind{Z_l>0}  e^{-\sum_{|u|=l} e^{-\beta V(u) \tilde{\mu}^{(u)}_{n,\beta}(F)}} \right] \leq \E\left[ e^{-\left(\alpha +\psi_\beta - \varepsilon\right)Z_l}\ind{Z_l>0}  \right].
\]
Letting $l \to \infty$ then $\varepsilon \to 0$, it leads to
\begin{equation}
  \label{eqn:sup}
  \limsup_{l \to \infty} \limsup_{n \to \infty} \E\left[ e^{-\alpha Z_l}\ind{Z_l>0}  e^{-\sum_{|u|=l} e^{-\beta V(u) \tilde{\mu}^{(u)}_{n,\beta}(F)}} \right] \leq \E\left[ e^{-\left(\alpha + C_\beta \E\left(F(\epsilon)^\frac{1}{\beta}\right)Z_\infty\right)}\ind{Z_\infty>0} \right]. 
\end{equation}

The lower bound follows from similar arguments. Applying once again the Markov property at time $l$,
\[
  \E\left[ e^{-\alpha Z_l}\ind{Z_l>0}  e^{-\sum_{|u|=l} e^{-\beta V(u) \tilde{\mu}^{(u)}_{n,\beta}(F)}} \right] \geq \E\left[ e^{-\alpha Z_l} \ind{Z_l>0, 3 M_l \geq \log l} \prod_{|u|=l} \Psi(V(u)) \right].
\]
Therefore,
\begin{align*}
    &\E\left[ e^{-\alpha Z_l}\ind{Z_l>0}  e^{-\sum_{|u|=l} e^{-\beta V(u) \tilde{\mu}^{(u)}_{n,\beta}(F)}} \right]\\
    &\qquad\qquad\qquad\qquad \geq \E\left[ e^{-\alpha Z_l}\ind{Z_l>0, 3 M_l \geq \log l} \prod_{u \in \Xi_{n,l}} \Psi(V(u)) \times \prod_{|u|=l, u \not \in \Xi_{n,l}} \Psi(V(u)) \right]\\
    &\qquad\qquad\qquad\qquad = \E\left[ e^{-\alpha Z_l}\ind{Z_l>0, 3 M_l \geq \log l} \prod_{u \in \Xi_{n,l}} \Psi(V(u)) \times \prod_{|u|=l, V(u) \geq \log n} \Psi(V(u)) \right].
\end{align*}
For $n \geq 1$ large enough, by~\cite[Proposition 2.1]{Mad11}, there exists $c>0$ such that
\begin{equation}
  \label{consMad11}
  \forall x \geq \log n, 1 - \Psi(x) \leq c x e^{-x}.
\end{equation}
Consequently, \eqref{constailprop} yields
\begin{align*}
  &\E\left[ e^{-\alpha Z_l}\ind{Z_l>0}  e^{-\sum_{|u|=l} e^{-\beta V(u) \tilde{\mu}^{(u)}_{n,\beta}(F)}} \right]\\
  \geq &\E\left[ e^{-\alpha Z_l}\ind{Z_l>0, 3 M_l \geq \log l} \prod_{|u|=l} \left( 1 - \left[ \ind{V(u) \leq \log n} \psi_\beta + c\ind{V(u) \geq\log  n} + \varepsilon\right] V(u)e^{-V(u)} \right) \right]\\
  \geq & \E\left[ e^{-\alpha Z_l}\ind{Z_l>0, 3 M_l \geq \log l} e^{-(1+\varepsilon) \sum_{|u|=l} \left( \ind{V(u) \leq \log n} \psi_\beta + c\ind{V(u) \geq\log  n} + \varepsilon\right) V(u)e^{-V(u)}} \right],
\end{align*}
for $l \geq 1$ large enough, as for any $x > 0$ small enough, $1-x \geq e^{-(1+\varepsilon)x}$. Letting $n \to \infty$, we have
\[
  \liminf_{n \to \infty} E\left[ e^{-\alpha Z_l}\ind{Z_l>0}  e^{-\sum_{|u|=l} e^{-\beta V(u) \tilde{\mu}^{(u)}_{n,\beta}(F)}} \right]
  \geq \E\left[ \ind{Z_l > 0, 3M_l \geq \log l} e^{-\left( \alpha + (1 + \varepsilon) (\psi_\beta + \varepsilon)\right)} \right].
\]
Finally, using the fact that $S = \{ Z_\infty > 0\}$ and that $\liminf_{l \to \infty} M_l/\log l > 1/3$, we obtain, letting $l \to \infty$ then $\varepsilon \to 0$
\[
  \liminf_{l\to\infty}\liminf_{n \to \infty}  \E\left[ e^{- \sum_{|u|=l} e^{-\beta V(u)} \tilde{\mu}^{(u)}_{n,\beta}(F)} e^{-\alpha Z_l}\ind{Z_l>0} \right] \geq \E\left[ e^{-\left[\alpha + C_\beta \E\left(F(\epsilon)^\frac{1}{\beta}\right)\right]Z_\infty}\ind{Z_\infty>0} \right].
\]
This last equation, as well as \eqref{eqn:sup}, completes the proof of \eqref{finnn}.
\end{proof}

We now prove that for any $F\in\calC_b(\calD)$, we have
\begin{equation}
\label{eq:conv}
 \mu_{n,\beta}(F) \wconv \sum_{k \in \N} p_k F(\epsilon^{(k)}), 
\end{equation}
where $(\epsilon^{(k)})$ is a sequence of i.i.d. normalized Brownian excursions, and $(p_k, k \in \N)$ follows an independent Poisson-Dirichlet distribution with parameter $(\frac{1}{\beta},0)$, which by \cite[Theorem 5.2]{Bil99} is enough to conclude the proof of Theorem \ref{thm:main}.

\begin{proof}[Proof of Theorem~\ref{thm:main}]
We recall that $\mu_{n,\beta}$ is defined on $S$ by $\mu_{n,\beta}(F) = \frac{\tilde{\mu}_{n,\beta}(F)}{\tilde{\mu}_{n,\beta}(1)}$, for $F \in \calC_b(\calD)$. To prove the convergence in law of $\mu_{n,\beta}$, we start by identifying the limit law of $\tilde{\mu}_{n,\beta}$.

Let $(\Delta_k, k \in \N)$ be a Poisson point process on $\R$ with intensity $e^x dx$ and $(\epsilon^{(k)}, k \in \N)$ be an independent sequence of i.i.d. normalized Brownian excursions, independent from the branching random walk. We introduce a random measure $\tilde{\mu}_{\infty,\beta}$ on $\calD$ by
\[
  \tilde{\mu}_{\infty,\beta}
  = Z_\infty^\beta \sum_{k=1}^{\infty} e^{-\beta (\Delta_k -c_{\star}(\beta))} \delta_{\epsilon^{(k)}},
\]
where $ c_{\star}(\beta) := \log \frac{C_\beta}{-\int_\R \left(e^{-e^{-\beta u}}-1\right) e^u du}$. We first prove that for any non-negative $F \in \calC_b(\calD)$,
\begin{equation}
  \label{eqn:cvtilde}
  \tilde{\mu}_{n,\beta}(F) \wconv \tilde{\mu}_{\infty,\beta}(F).
\end{equation}

We compute the Laplace transform of $\tilde{\mu}_{\infty,\beta}(F)$. As $S = \{ Z_\infty > 0\}$, for any $\theta > 0$, we have
\begin{align*}  
  \E\left[ \exp\left( - \theta \tilde{\mu}_{\infty,\beta}(F) \right)\ind{S} \right]
  &=\E\left[\exp\left(-\theta Z_\infty^\beta \sum_{k=1}^{\infty} e^{-\beta [\Delta_k -c_{\star}(\beta)]} F(\epsilon^{(k)})\right)\ind{Z_\infty>0}\right]\\
  &= \E\left[ \exp\left( - \sum_{k = 1}^{\infty} \phi \left(\theta e^{-\beta [\Delta_k-\log Z_\infty - c_{\star}(\beta)]} \right)  \right)\ind{Z_\infty>0} \right],
\end{align*}
where $\phi : x \mapsto \log \E\left[ \exp\left(-uF(\epsilon^{(1)})\right) \right]$. By Campbell's formula, 
\[
\E\left[ \exp\left( - \theta \tilde{\mu}_{\infty,\beta}(F) \right)\ind{S} \right]= \E\left[ \exp\left( \int_\R e^{\phi\left( \theta e^{-\beta [x-\log Z_\infty - c_{\star}(\beta)]}\right)+x-1} dx \right)\ind{Z_\infty>0} \right].
\]
As $\phi\left( \theta e^{-\beta [x-\log Z_\infty - c_{\star}(\beta)]}\right)= \log \E\left[\left. \exp\left(-\theta e^{-\beta [x-\log Z_\infty - c_{\star}(\beta)]}F(\epsilon^{(1)})\right)\right| Z_\infty \right]$, it yields
\begin{align*}
&\E\left[ \exp\left( - \theta \tilde{\mu}_{\infty,\beta}(F) \right) \ind{S}\right]\\
&=\E\left[ \exp\left( \int_\R \E\left[ \exp\left(-\theta e^{-\beta [x-\log Z_\infty - c_{\star}(\beta)]}F(\epsilon^{(1)})\right)\Big\vert Z_\infty \right]-1\right) e^x dx \ind{Z_\infty>0}\right]\\
&=\E\left[ \exp\left(\E\left[ \int_\R \left(\exp\left(-\theta e^{-\beta [x-\log Z_\infty - c_{\star}(\beta)]}F(\epsilon^{(1)})\right)-1\right) e^x dx \Big\vert Z_\infty\right] \right)\ind{Z_\infty>0}\right].
\end{align*}
By change of variables $u=x-\log Z_\infty - c_{\star}(\beta) - \frac{1}{\beta}\left[\log \theta +\log F(\epsilon^{(1)})\right]$, we obtain that
\begin{align*}
&\E\left[ \int_\R \left(\exp\left(-\theta e^{-\beta [x-\log Z_\infty - c_{\star}(\beta)]}F(\epsilon^{(1)})\right)-1\right) e^x dx \Big\vert Z_\infty\right]\\
 =&\E\left[ \int_\R \left(e^{- e^{-\beta u}}-1\right) e^{u+\log Z_\infty+c_{\star}(\beta)+\frac{1}{\beta}\left[\log \theta +\log F(\epsilon^{(1)})\right]} du \Big\vert Z_\infty\right] \\
 = &e^{c_{\star}(\beta)}Z_\infty \E\left[\left(\theta F(\epsilon^{(1)})\right)^\frac{1}{\beta}\right]\int_\R \left(e^{-e^{-\beta u}}-1\right) e^u du =-C_\beta Z_\infty \E\left( (F(\epsilon^{(1)})\theta)^\frac{1}{\beta} \right),
\end{align*}
since $ c_{\star}(\beta) = \log \frac{C_\beta}{-\int_\R \left(e^{-e^{-\beta u}}-1\right) e^u du}$. We thus end up with
\[
\E\left[ \exp\left( - \theta \tilde{\mu}_{\infty,\beta}(F) \right)\ind{S} \right]=\E\left[ \exp\big\{-C_\beta Z_\infty \E\left( (F(\epsilon^{(1)})\theta)^\frac{1}{\beta} \right)\big\}\ind{Z_\infty>0}\right].
\]

Consequently, by Proposition~\ref{propfi}, for any $\theta > 0$ we have
\[
  \lim_{n \to \infty} \E\left[ \ind{Z_\infty > 0} e^{-\theta \tilde{\mu}_{n,\beta}(F)} \right] = \E\left[  \ind{Z_\infty > 0} e^{-\theta \tilde{\mu}_{\infty,\beta}(F)} \right],
\]
which concludes \eqref{eqn:cvtilde} by L\'evy's theorem.

Furthermore, for any $F \in \calC_b(\calD)$ and $\theta_1,\theta_2,\theta_3 \in (0,\infty)$, setting $F_+$ (respectively $F_-$) the positive (resp. negative) part of $F$, we have
\[
  \lim_{n \to \infty} \E\left[ \ind{Z_\infty > 0} e^{-\theta \tilde{\mu}_{n,\beta}(\theta_1 F_+ + \theta_2 F_- + \theta_3)} \right] = \E\left[  \ind{Z_\infty > 0} e^{-\theta \tilde{\mu}_{\infty,\beta}(\theta_1 F_+ + \theta_2 F_- + \theta_3)} \right],
\]
yielding
\[
  \left(\tilde{\mu}_{n,\beta}(F_+), \tilde{\mu}_{n,\beta}(F_-), \tilde{\mu}_{n,\beta}(1)\right) \wconv \left(\tilde{\mu}_{\infty,\beta}(F_+), \tilde{\mu}_{\infty,\beta}(F_-), \tilde{\mu}_{\infty,\beta}(1)\right).
\]
Using the fact that $\tilde{\mu}_{\infty,\beta}(1)>0$ a.s. on $S$, we have
\[
  \mu_{n,\beta}(F) = \frac{\tilde{\mu}_{n,\beta}(F_+) - \tilde{\mu}_{n,\beta}(F_-)}{\tilde{\mu}_{n,\beta}(1)} \underset{n \to \infty}{\Longrightarrow} \sum_{k\geq0}\frac{e^{-\beta\Delta_k}}{\sum_{j=0}^{\infty} e^{-\beta \Delta_j}}F(\epsilon^{(k)})\quad \mathrm{on} \quad S.
\]
\begin{remark}
\label{rem:inthm}
We obtain in a similar way the joint convergence of $\left( \mu_{n,\beta}(F_1), \cdots \mu_{n,\beta}(F_k)\right)$, for any $(F_1,\ldots, F_k) \in \calC_b(\calD)^k$.
\end{remark}

Using~\cite[Proposition 10]{PiY97}, for a Poisson point process $(\Delta_k, k\geq 0)$ of intensity $e^x$, we have
\[
 \left( \frac{e^{-\beta\Delta_k}}{\sum_{j=0}^{\infty} e^{-\beta \Delta_j}}, k \geq 0 \right) \egaldistr (p_k, k \geq 0),
\]
where $(p_k)$ is a process of Poisson-Dirichlet distribution with parameters $(\frac{1}{\beta},0)$.
\end{proof}

We conclude this article proving that Theorem~\ref{thm:main} implies Corollary \ref{cor:interest}.
\begin{proof}[Proof of Corollary \ref{cor:interest}]
We first observe that for all $t \in (0,1)$, the function
\[
  \phi_t : (f,g) \in \calD^2 \mapsto \ind{\forall s < t, f(s) = g(s)}
\]
is almost everywhere continuous for the law $\mu_{\infty,\beta}^{\otimes 2}$. Therefore, by Portmanteau theorem, we have
\[
  \lim_{n \to \infty} \mu_{n,\beta}^{\otimes 2}(\phi_t) = \mu_{\infty,\beta}^{\otimes 2} (\phi_t) = \sum_{k \geq 0} p_k^2 \quad \text{in law,}
\]
by Theorem \ref{thm:main}. This proves that for all $t \in (0,1)$
\[
  \lim_{n \to \infty} \mu_{n,\beta}^{\otimes 2}\left( \inf\{ s > 0 : V(z_\floor{ns}) \neq V(z'_\floor{ns}\} \geq t \right)  = \rho_\beta \quad \text{ in law}.
\]

We now prove this convergence in law holds simultaneously for all $t \in (0,1)$. For any $n \in \N$ and $t \in (0,1)$, set
\[
  F^{(n,\beta)}(t) = \mu_{n,\beta}^{\otimes 2}\left( \inf\{ s > 0 : V(z_\floor{ns}) \neq V(z'_\floor{ns}\} \leq t \right).
\]
By monotonicity of repartition functions, we have $F^{(n,\beta)}(t) \geq F^{(n,\beta)}(s)$ for all $t \geq s$. Moreover, by dominated convergence theorem, we have
\[
  \lim_{n \to \infty} \E(F^{(n,\beta)}(t) - F^{(n,\beta)}(s)) = \E(F^{(\infty,\beta)}(t)) - \E(F^{(\infty,\beta)}(s)) = 0.
\]
As a result, $F^{(n,\beta)}(t) - F^{(n,\beta)}(s)$ being a non-negative random variable, this implies that it converges to 0 in probability. Therefore, by Slutsky's lemma, we have
\[
  \lim_{n \to \infty} (F^{(n,\beta)}(t),F^{(n,\beta)}(t) - F^{(n,\beta)}(s)) = (1 - \rho_\beta,0) \quad \text{ in law.}
\]
With this line of reasoning, we prove that the convergence of $F^{(n,\beta)}(t)$ to $1-\rho_\beta$ holds simultaneously for all $t \in (0,1) \cap \mathbb{Q}$. Therefore, we obtain the convergence
\[
\lim_{n \to \infty} \mu_{n,\beta}^{\otimes 2}\left( \inf\{ t > 0 : V(z_\floor{nt}) \neq V(z'_\floor{nt}\} \in dx \right)  =  \rho_\beta \delta_1 + (1-\rho_\beta)\delta_0 \quad \text{ in law},
\]
which concludes the proof.
\end{proof}

\bibliographystyle{plain}

\end{document}